\documentclass[10pt,a4paper]{article}
\usepackage[english]{babel}
\usepackage[margin=1in]{geometry}

\usepackage{amsmath,amsthm,amssymb,enumerate, appendix, bm, colortbl, xcolor, longtable, float}
\usepackage[pdftex]{hyperref} 

\newcommand{\E}{\mathbb{E}}

\newcommand{\C}{\mathbb{C}}
\newcommand{\R}{\mathbb{R}}

\newcommand{\one}{\mathbf{1}}

\newtheorem{theorem}{Theorem}[section]
\newtheorem{lemma}[theorem]{Lemma}

\newtheorem{proposition}[theorem]{Proposition}

\theoremstyle{definition}
\newtheorem{remark}{Remark}[section]

\newcommand{\xx}{\mathbf{x}}
\newcommand{\yy}{\mathbf{y}}
\newcommand{\uu}{\mathbf{u}}
\newcommand{\vv}{\mathbf{v}}
\newcommand{\ww}{\mathbf{w}}
\newcommand{\zero}{\mathbf{0}}
\newcommand{\inner}[2]{\langle#1,#2\rangle}
\newcommand{\Aut}{\mathrm{Aut}}
\newcommand{\AutZ}{\mathrm{Aut}_{\zero}}
\newcommand{\T}{^{\mathsf{T}}}

\DeclareMathOperator*{\diag}{diag}
\DeclareMathOperator*{\tr}{tr}
\DeclareMathOperator*{\supp}{supp}

\begin{document}
\title{Semidefinite lower bounds for covering codes}
\author{Dion Gijswijt\thanks{Delft Institute of Applied Mathematics, Delft University of Technology, The Netherlands. E-mail: \href{d.c.gijswijt@tudelft.nl}{\texttt{d.c.gijswijt@tudelft.nl}}}\,\, and Sven Polak\thanks{Tilburg University, The Netherlands. E-mail: \href{s.c.polak@tilburguniversity.edu}{\texttt{s.c.polak@tilburguniversity.edu}}}}
\maketitle

\abstract{Let~$K_q(n,r)$ denote the minimum size of a $q$-ary \emph{covering code} of word length~$n$ and covering radius~$r$. In other words,~$K_q(n,r)$ is the minimum size of a set of~$q$-ary codewords of length~$n$ such that the Hamming balls of radius~$r$ around the codewords cover the Hamming space~$\{0,\ldots,q-1\}^n$. The special case~$K_3(n,1)$ is often referred to as the \emph{football pool problem}, as it is equivalent to finding a set of forecasts on~$n$ football matches that is guaranteed to contain a forecast with at most one wrong outcome. 

In this paper, we build and expand upon the work of Gijswijt (2005), who introduced a semidefinite programming lower bound on~$K_q(n,r)$ via matrix cuts. We develop techniques that strengthen this bound, by introducing new semidefinite constraints inspired by Lasserre's hierarchy for 0-1 programs and symmetry reduction methods, and a more powerful objective function. The techniques lead to sharper lower bounds, setting new records across a broad range of values of~$q$, $n$, and~$r$.}

\section{Introduction}
 Let~$q \geq 2$ and~$n\geq 1$ be integers. Then~$[q]^n := \{0,\ldots,q-1\}^n$ is the \emph{Hamming space} of~$q$-ary words of length~$n$. The \emph{Hamming distance} between $\uu,\vv\in [q]^n$ is denoted by $d(\uu,\vv)=|\{i\in\{1,\ldots,n\}: \uu_i\neq \vv_i\}|$. We denote the isometry group of $[q]^n$ by $\Aut(q,n)$. It consists of all permutations of $[q]^n$ obtained by permuting the $n$ positions and independently permuting the $q$ symbols in each of the $n$ positions. The group $\Aut(q,n)$ acts on subsets $S\subseteq [q]^n$, vectors $\xx \in\C^{[q]^n}$ and matrices $M\in \C^{[q]^n\times [q]^n}$ in the usual way by
\[
\sigma S=\{\sigma(a)\mid a\in S\},\quad (\sigma \xx)_{\uu}=\xx_{\sigma^{-1}(\uu)}, \quad (\sigma M)_{\uu,\vv}=M_{\sigma^{-1}(\uu),\sigma^{-1}(\vv)}
\]
for all $\sigma\in \Aut(q,n)$ and $\uu, \vv \in [q]^n$. For any integer~$0\leq r\leq n$ and any $\uu\in[q]^n$, we denote by 
\begin{align*}
B_r(\uu)&:=\{\vv\in [q]^n \mid d(\uu,\vv)\leq r\}\quad\text{and}\\
S_r(\uu)&:=\{\vv\in [q]^n \mid d(\uu,\vv)=r\} \notag 
\end{align*}
the \emph{Hamming ball} and \emph{Hamming sphere} of radius $r$ around $\uu$, respectively.

A $q$-ary \emph{code} of length~$n$ is a subset $C$ of $[q]^n$. The \emph{covering radius} of $C$ is the smallest integer~$r$ such that the Hamming balls of radius~$r$ around the codewords cover the Hamming space~$[q]^n$, i.e., $\bigcup_{\uu \in C} B_r(\uu) = [q]^n$. In this paper, we study the quantity 
\begin{align}
    K_q(n,r) := \text{min}\{ |C| \, \mid \, C \subseteq [q]^n \text{ has covering radius~$r$} \}. 
\end{align}

In other words, the problem of determining $K_q(n,r)$ amounts to finding the
smallest number of Hamming balls of radius~$r$ that cover the Hamming space~$[q]^n$. If $C\subseteq [q]^n$ has covering radius $r$ and size $K=|C|$, then we call $C$ an $(n,K,q,r)$-code. 

A popular application of covering codes is found in football pools. A football pool involves predicting the outcomes of $n$ football matches, where each match can result in three possible outcomes: a home team win, a draw, or an away team win. The objective is to find the smallest set of bets that is guaranteed to contain a bet with at most one incorrect outcome. This challenge, commonly known as the football pool problem, corresponds to the covering problem for $q = 3$ and $r = 1$.  For more details on the football pool problem, see~\cite{H95}. For an overview of results on covering codes and several of their applications, see~\cite{CHLL97}.

A fundamental problem in the study of covering codes is to determine tight bounds on $K_q(n,r)$. Upper bounds are commonly established by constructing an explicit code with covering radius~$r$. In this paper, we focus on establishing lower bounds on~$K_q(n,r)$.

\paragraph{Method of linear inequalities}
A key approach for deriving lower bounds on $K_q(n, r)$ is the method of linear inequalities. Here, we follow the exposition in~\cite[Chapter 6]{CHLL97}. Suppose we have a code $C \subseteq [q]^n $ and define  
\[
A_i(\uu) := |C \cap S_i(\uu)|  
\]
for any $\uu \in [q]^n$ and $i = 0, \ldots, n$. The method considers valid linear inequalities over a code that are of the form 
\begin{align}\label{eq:lineqofcode}
\sum_{i=0}^n \lambda_i A_i(\uu) \geq \beta \quad \text{for all } \uu \in [q]^n,  
\end{align}
where $\lambda_0, \ldots, \lambda_n \geq 0$ and $\beta > 0$. Such a set of inequalities is denoted by $(\lambda_0, \ldots, \lambda_n) \beta$. The following proposition shows how such a set leads to a lower bound on $ K_q(n, r) $.

\begin{proposition} \label{prop:lowerboundineq}
If every $(n, K, q, r)$-code satisfies the inequalities $(\lambda_0, \ldots, \lambda_n) \beta$, then  
\[
K \geq \frac{\beta q^n}{\sum_{i=0}^n \lambda_i \binom{n}{i} (q - 1)^i}.  
\] 
\end{proposition}  

\begin{proof}  
By summing the inequalities in~$\eqref{eq:lineqofcode}$ over all $\uu \in [q]^n$, we obtain  
\begin{align*}  
\beta q^n \leq \sum_{\uu \in [q]^n} \sum_{i=0}^n \lambda_i A_i(\uu)
&= \sum_{i=0}^n \lambda_i \sum_{\uu \in [q]^n} A_i(\uu)
= \sum_{i=0}^n \lambda_i \sum_{\vv \in C} |S_i(\vv)| 
= |C| \sum_{i=0}^n \lambda_i \tbinom{n}{i} (q - 1)^i. \qedhere  
\end{align*}  
\end{proof}  
The \emph{sphere covering inequalities} of the form  
\begin{align}
\sum_{i=0}^r A_i(\uu) \geq 1 \quad \text{for all } \uu \in [q]^n 
\end{align}
in combination with Proposition~\ref{prop:lowerboundineq} lead to the \emph{sphere covering bound}, given by  
\begin{align}
K_q(n, r) \geq \frac{q^n}{\sum_{i=0}^r \tbinom{n}{i} (q - 1)^i}.  
\end{align}

Several additional valid inequalities have been derived in the literature, especially in the binary case ($q = 2$), by analyzing how elements of $B_s(\uu)$ can be covered for $s = 1, 2, 3$. For $s = 1$, this results in the Van Wee inequalities \cite{VanWee88, VanWeeThesis}, given by:  
\begin{align}
\sum_{i=0}^{r-1} \left\lceil \frac{n+1}{r+1} \right\rceil A_i(\uu) + A_r(\uu) + A_{r+1}(\uu) \geq \left\lceil \frac{n+1}{r+1} \right\rceil  
\end{align}
which gives an improvement over the sphere covering bound when $r + 1$ does not divide $n + 1$.  
Additional inequalities can be found in articles by Johnson~\cite{Johnson72}, Zhang~\cite{Zhang91}, and Zhang and Lo~\cite{ZL92}.  

By taking nonnegative linear combinations of inequalities, for instance by summing over sets $S_i(\uu)$, new inequalities can be derived. Utilizing the fact that the $A_i(\uu)$ are integers, $(\lambda_0, \ldots, \lambda_n) \beta$ implies $(\lceil \lambda_0 \rceil, \ldots, \lceil \lambda_n \rceil) \lceil \beta \rceil$. This way, strengthened bounds can be found. For instance, the van Wee inequalities can be deduced from the sphere covering inequalities by summing over $B_1(\uu)$, multiplying by $\tfrac{1}{r+1}$ and rounding up. Habsieger and Plagne extended this method to generate numerous new lower bounds for binary and ternary covering codes~\cite{HP00}.

\paragraph{Semidefinite programming bounds}
In \cite[Section 5.3.1]{DionThesis} a linear programming lower bound for $K_q(n,r)$ is given that is analogous to the Delsarte bound for error correcting codes~\cite{Del73}. 

If every code $C\subseteq [q]^n$ with covering radius $r$ satisfies the inequalities $(\lambda_0,\ldots,\lambda_n)\beta$, then 
\[
K_q(n,r) \geq \min_x q^n x_0,
\]
where the minimum ranges over all nonnegative vectors $x=(x_0,x_1,\ldots,x_n)^{\sf T} \in \R^{n+1}$ that satisfy
\begin{enumerate}[(i)]
\item $\sum_{i=0}^n x_i P_k(i) \geq 0$,
\item $\sum_{i=0}^n x_i \sum_{j=0}^n \lambda_j \alpha_{i,j}^k \geq \beta x_0$,
\item $\sum_{i=0}^n (x_0-x_i) \sum_{j=0}^n \lambda_j \alpha_{i,j}^k \geq \beta (1-x_0)$,
\end{enumerate}
for all~$k=0,\ldots,n$. Here, $P_k(i):= \sum_{s=0}^k (-1)^s \tbinom{i}{s}\tbinom{n-i}{k-s}(q-1)^{k-s}$ are the \emph{Krawtchouk polynomials}, and the numbers $\alpha_{i, j}^k$ are given by $\alpha_{i, j}^k= |\{ \mathbf{v}\in[q]^n \mid d(\mathbf{0}, \mathbf{v}) = i,\, d(\mathbf{v}, \mathbf{u}) = j \}|$   for any~$\uu \in[q]^n$ with $d(\mathbf{0}, \mathbf{u}) = k$, and can be expressed as
\[
\alpha_{i, j}^k = \begin{cases}  
\sum_{\substack{p, t \\ t + p = k + i - j}} \tbinom{k}{t - p, p} \tbinom{n - k }{i - t} (q - 1)^{i - t} (q - 2)^{t - p} & \text{if } q \geq 3, \\
\sum_{\substack{t \\ 2t = k + i - j}} \tbinom{k}{t} \tbinom{n - k}{i - t} & \text{if } q = 2.  
\end{cases}  
\]

The $n+1$ variables~$x_i$ correspond to the distance distribution of the code. Where in the Delsarte bound variables corresponding to forbidden distances are set to zero, the bound on covering codes has inequalities (ii) and (iii) on the variables induced by the sphere covering inequalities, or any other valid set of inequalities $(\lambda_0,\ldots, \lambda_n)\beta$. 

The Delsarte bound can be viewed as a semidefinite program 
where the $[q]^n\times [q]^n$ matrix variable is constrained to be invariant under $\Aut(q,n)$ (i.e., it is an element of the Bose-Mesner algebra of the Hamming scheme). Since the Bose-Mesner algebra has dimension $n+1$, the matrix variable can be expressed in terms of $n+1$ real variables (corresponding to the distance distribution of the code) and because the Bose-Mesner algebra is commutative, the semidefiniteness constraint can be reduced to linear constraints in the $n+1$ variables by simultaneously diagonalizing the matrices in the Bose-Mesner algebra. This corresponds to inequalities (i).

For error correcting codes, strengthened SDP bounds based on the distribution of triples~\cite{Sch05,GST06} of codewords and four-tuples~\cite{GMS12,LPS2017} have been obtained. The main complication for these bounds is that the matrix variable is now an element of a non-commutative algebra. In the case of triple bounds, the relevant algebra consists of the $[q]^n\times [q]^n$ matrices that are invariant under $\AutZ(q,n)$, the stabilizer subgroup with respect to the zero word. This algebra coincides with the Terwilliger algebra of the Hamming scheme. To reduce the SDP to an equivalent one of polynomial size, an explicit block diagonalization of this algebra had to be obtained. 

An analogous bound for $K_q(n,r)$ based on triples of codewords, was obtained by Gijswijt in~\cite[Section 5.3.2]{DionThesis} using the method of matrix cuts from~\cite{LS}. Here, we develop techniques to strengthen this bound. We introduce new semidefinite constraints inspired by the Lasserre Hierarchy for 0-1 programs~\cite{Lasserre, Lasserre2001b} and symmetry reduction, alongside a more powerful objective function. The techniques lead to sharper lower bounds, improving the best known bounds for a broad range of values of~$q$, $n$, and~$r$.

The bound we study in this paper is a covering analogue to Schrijver's first SDP bounds for error-correcting codes in the finite Hamming space~\cite{Sch05}. Such bounds can be interpreted as arising from submatrices corresponding to 3-point configurations in the second level of the Lasserre hierarchy~\cite{Lasserre} applied to a polynomial optimization formulation of the problem. In the context of discrete geometry and spherical codes, Riener, Rolfes, and Vallentin~\cite{RRV26} recently proposed a hierarchy for spherical covering codes, inspired by hierarchies for packing problems in discrete geometry~\cite{LV15}. In our work, we extend Schrijver's packing bound to covering codes in the finite Hamming space. While in principle one could also define a full hierarchy for covering codes in this setting, practical obstacles (in particular, the rapidly increasing number of variables and the size of the block matrices) make the full second and higher levels of the Lasserre hierarchy difficult to compute. Therefore, we focus on a 3-point bound and its symmetry reduction, and strengthen it as much as we can by introducing additional constraints and optimizing the choice of objective function.

\paragraph{Outline of the paper}
In Section~\ref{sec:generalsdp} we describe our SDP bound for covering codes in Theorem~\ref{thm:originalbound}, before applying symmetry reduction. In Section~\ref{sec:terwilliger} we describe the Terwilliger algebra of the binary and nonbinary Hamming schemes, which enables symmetry reduction of the bound from Section~\ref{sec:generalsdp}. In Section~\ref{sec:binary} we detail the symmetry reduction in the binary case ($q=2$) and present the reduced bound in Theorem~\ref{thm:sdplowerbound}. Section~\ref{sec:nonbinary} provides the reduced version of our SDP bound for the nonbinary case in Theorem~\ref{thm:sdplowerboundqary}. 
Finally, we present our improved bounds in tables in Section~\ref{sec:tables}, with expanded versions of our numerical results in Appendices~\ref{appendix:tablesbinary} and~\ref{appendix:tablesqary}.

\section{The SDP bound for covering codes \label{sec:generalsdp}}
Given a complex matrix $A$, we write $A\T$ for the transpose of $A$ and $A^*$ for the complex conjugate transpose of $A$. So $A^*=A\T$ if $A$ is a real matrix. We apply the same notation for vectors and scalars. So for $z\in \C$ we denote the complex conjugate of $z$ by $z^*$. Given complex matrices $A,B$ of the same size, we denote by $\inner{A}{B}=\tr (A^*B)$ their Frobenius inner product. If $A$ is a square matrix, we write $A\succeq 0$ to denote that $A$ is positive semidefinite. Throughout the paper, we write $J$ for the all-ones matrix, and $I$ for the identity matrix, where the dimensions are clear from the context.

Given a square matrix $A$ and $c\geq 0$ we will use the following notation:
\[
R(c,A)=\begin{pmatrix}c&(\diag A)^*\\\diag A&A \end{pmatrix}\quad \text{and}\quad R(A)=R(1,A).
\]
Note that for positive semidefinite $A$, the matrix $R(c,A)$ is PSD if and only if either $c=0$ and $A=0$, or $c > 0$ and $cA - (\diag A)(\diag A)^* \succeq 0$, which follows by taking Schur complements. 
We remark that throughout the paper, we will mostly work with real matrices. In particular, the matrices in our SDP bounds will be real valued.  

Let $q\geq 2$ and $n\geq 1$ be integers, and denote $\E=[q]^n$. We identify $[q]^n$ with~$(\mathbb{Z}/q\mathbb{Z})^n$, so that $\E $ carries an additive group structure. Let $C\subseteq \E$. We define the $0,1$-matrix~$M_C$ of size~$|\E| \times |\E|$ as follows:
\[
(M_C)_{\uu,\vv} = \begin{cases} 1 &\text{if $\uu,\vv \in C$,} \\0 &\text{otherwise.}\end{cases} 
\]
Furthermore, we define the following real matrices associated to $C$:
\begin{align}
 M&:= |\Aut(q,n)|^{-1} \sum_{\sigma \in \Aut(q,n) } M_{\sigma C}, \notag \\
M'&:= |\Aut(q,n)|^{-1} \sum_{\substack{\sigma \in \Aut(q,n)\\ \zero \in \sigma C }} M_{\sigma C}, \quad \text{and} \quad
M'':= |\Aut(q,n)|^{-1} \sum_{\substack{\sigma \in \Aut(q,n)\\ \zero \not\in \sigma C }} M_{\sigma C}.
\end{align}

We immediately observe the following relations for the entries of the matrices $M,M',M''$.
\begin{proposition}[Basic inequalities and symmetry]\label{prop:first}
The matrix $M$ is $\Aut(q,n)$-invariant and $M',M''$ are $\AutZ(q,n)$-invariant. Moreover, for all $\uu,\vv\in\E$ we have
\begin{itemize}
\item[\textup{(i)}] $M_{\uu,\vv}=M'_{\zero,\vv-\uu}$ \quad and \quad $M''_{\uu,\vv}=M'_{\zero,\vv-\uu}-M'_{\uu,\vv}$,
\item[\textup{(ii)}] $0 \leq M'_{\uu,\vv} \leq M'_{\zero,\uu}$,
\item[\textup{(iii)}] $0 \leq M''_{\uu,\vv} \leq M''_{\uu,\uu}$, 
\item[\textup{(iv)}] $M'_{\uu,\vv}=M'_{\uu',\vv'}$ if $\{\zero,\uu',\vv'\}=\sigma\{\zero,\uu,\vv\}$ for some $\sigma\in \Aut(q,n)$.
\end{itemize}
\end{proposition}

\begin{proof}
The fact that $M,M',M''$ are invariant under the respective groups follows directly by their construction. To see (i), note that $M''_{\zero,\ww}=0$ for every $\ww\in \E$, so 
\[
M_{\uu,\vv}=M_{\zero, \vv-\uu}=M'_{\zero,\vv-\uu}+M''_{\zero,\vv-\uu}=M'_{\zero,\vv-\uu}
\]
and 
\[
M''_{\uu,\vv}=M_{\uu,\vv}-M'_{\uu,\vv}=M'_{\zero,\vv-\uu}-M'_{\uu,\vv}.
\]
To see inequalities (ii) and (iii), note that $M'$ and $M''$ are nonnegative linear combinations of the matrices $M_{\sigma C}$ and $0\leq \left(M_{\sigma C}\right)_{\uu,\vv}\leq\left(M_{\sigma C}\right)_{\uu,\uu}$. 

To see (iv), we observe that
\[
M'_{\uu,\vv}=|\Aut(q,n)|^{-1}\cdot |\{\sigma\in \Aut(q,n)\mid \{\zero,\uu,\vv\}\subseteq \sigma C\}|
\]
only depends on the orbit of $\{\zero,\uu,\vv\}$.
\end{proof}

The matrices $M'$ and $M''$ satisfy the following conditions.
\begin{proposition}[Semidefiniteness]\label{prop:second}
The matrices $M',M''$ are positive semidefinite. Moreover, the matrix $R(1-M'_{\zero,\zero},M'')$ is positive semidefinite.
\end{proposition}
\begin{proof}
The matrices $M'$ and $M''$ are nonnegative linear combinations of the positive semidefinite matrices $M_{\sigma C}=\one_{\sigma C}(\one_{\sigma C})\T$, and therefore themselves positive semidefinite. Since the matrices $R(M_{\sigma C})$ are positive semidefinite, so is 
\[
R(1-M'_{\zero,\zero},M'')=\begin{pmatrix}1-M'_{\zero,\zero}&\left(\diag M''\right)\T\\\diag M''&M''\end{pmatrix}=|\Aut(q,n)|^{-1} \sum_{\substack{\sigma \in \Aut(q,n)\\ \zero \not\in \sigma C }} R(M_{\sigma C}). \qedhere
\]
\end{proof}

The cardinality of $C$ can be expressed in terms of $M'$ in several ways as follows.
\begin{proposition}[Objective function]
We have 
\[|C|=q^nM'_{\zero,\zero},\quad |C|^2=q^n\sum_{\uu\in\E}M'_{\uu,\uu},\quad |C|^3=q^n\sum_{\uu,\vv\in \E}M'_{\uu,\vv}.
\]
\end{proposition}
\begin{proof}
For any~$\uu \in \E$ we have $|\{\sigma \in \Aut(q,n) \, | \, \sigma \uu = \zero\}| = |\AutZ(q,n)|$, so
\[
M'_{\zero,\zero} = \frac{|\{\sigma \in \Aut(q,n) \, | \,  \zero \in \sigma C\}|}{|\Aut(q,n)|} =|C|\frac{|\AutZ(q,n)|}{|\Aut(q,n)|} =q^{-n}|C|.
\]
Moreover, 
\begin{align*}
\sum_{\uu\in\E}M'_{\uu,\uu} &= \inner{M'}{I} = |\Aut(q,n)|^{-1} \sum_{\uu \in C} \sum_{\substack{\sigma \in \Aut(q,n)\\ \sigma \uu = \zero}} \inner{M_{\sigma C}}{I} =|\Aut(q,n)|^{-1} \sum_{\uu \in C} \sum_{\substack{\sigma \in \Aut(q,n)\\ \sigma \uu = \zero}} |C| 
\\&=\frac{|\AutZ(q,n)|}{|\Aut(q,n)|}|C|^2 =q^{-n}|C|^2,
\end{align*}
and similarly
\begin{align*}
 \sum_{\uu, \vv \in\E}M'_{\uu,\vv} &= \inner{M'}{J}  = |\Aut(q,n)|^{-1} \sum_{\uu \in C} \sum_{\substack{\sigma \in \Aut(q,n)\\ \sigma \uu = \zero}} \inner{M_{\sigma C}}{J} =|\Aut(q,n)|^{-1} \sum_{\uu \in C} \sum_{\substack{\sigma \in \Aut(q,n)\\ \sigma \uu = \zero}} |C|^2 \\&=q^{-n}|C|^3.
 \qedhere 
\end{align*}
\end{proof}

Now suppose that $C$ has covering radius $r$ and that the set of inequalities $(\lambda_0,\ldots, \lambda_n)\beta$ is valid for all codes of covering radius $r$.  Define
\begin{align} \label{eq:Nmatrix}
N=|\Aut(q,n)|^{-1} \sum_{\sigma \in \Aut(q,n)} M_{\sigma C}\cdot\left(\sum_{\ell=0}^n\lambda_\ell |\sigma C\cap S_{\ell}(\zero)|-\beta\right).
\end{align}

\begin{proposition}[Lasserre constraint and matrix cut inequalities]\label{prop:third}
The following hold:
\begin{itemize}
\item[\textup{(i)}] The matrix $R(c,N)$ is positive semidefinite, where $c=\sum_{\ell=0}^n \lambda_{\ell}\cdot|S_{\ell}(\zero)|\cdot M'_{\zero,\zero}
 - \beta$,
\item[\textup{(ii)}] $N$ is invariant under $\AutZ(q,n)$,
\item[\textup{(iii)}] For all $\uu,\vv\in \E$ we have 
\[
N_{\uu,\vv}=-\beta M_{\uu,\vv}+\sum_{\ell=0}^n \lambda_\ell\sum_{\ww\in S_{\ell}(\zero)} M'_{\uu-\ww,\vv-\ww},
\]
\item[\textup{(iv)}] For all $\uu,\vv \in \E$ we have
\begin{align*}
\beta M'_{\zero,\uu}
&\leq \sum_{\ell=0}^n \lambda_\ell \sum_{\ww\in S_{\ell}(\vv)} M'_{\uu,\ww},\\
\beta \bigl(M'_{\zero,\zero}-M'_{\zero,\uu}\bigr)
&\leq \sum_{\ell=0}^n \lambda_\ell \sum_{\ww\in S_{\ell}(\vv)} \bigl(M'_{\ww,\ww}-M'_{\uu,\ww}\bigr),\\
\beta M''_{\uu,\uu}
&\leq \sum_{\ell=0}^n \lambda_\ell \sum_{\ww\in S_{\ell}(\vv)} M''_{\uu,\ww},\\
\beta \bigl((1-M'_{\zero,\zero})-M''_{\uu,\uu}\bigr)
&\leq \sum_{\ell=0}^n \lambda_\ell \sum_{\ww\in S_{\ell}(\vv)} \bigl(M''_{\ww,\ww}-M''_{\uu,\ww}\bigr).
\end{align*}
\end{itemize}
\end{proposition}

\begin{remark}
The conditions in Proposition~\ref{prop:third} derive from general relaxations for 0-1 linear programs. In our setting, the incidence vector $\xx\in \{0,1\}^\E$ of $C$ satisfies the covering inequalities 
\[
\sum_{i=0}^n \lambda_i \sum_{\vv\in S_i(\uu)}\xx_\vv\geq \beta\quad \text{for all $\uu\in \E$.}
\]
However, the same method would apply in the setting of error correcting codes, where one might use inequalities $\xx_u+\xx_v\leq 1$, if $d(\uu,\vv)\in \{1,\ldots, d-1\}$. 

The matrix $N$ and conditions (i)--(iii) come from the Lasserre Hierarchy for $0,1$ polytopes. Let~$\yy$ be the extended incidence vector of $C$, where for every subset $S\subseteq \E$ we set $\yy(S)=1$ if $S\subseteq C$ and $\yy(S)=0$ otherwise. The validity of the covering inequalities implies that for every $\uu\in\E$, the matrix $M\in \C^{\mathcal{P}(\E)\times \mathcal{P}(\E)}$ given by $M_{I,J}=\sum_{i=0}^n \lambda_i\sum_{\vv\in S_i(\uu)} y_{I\cup J\cup\{\vv\}}-\beta y_{I\cup J}$ is positive semidefinite. Restricting to sets $I,J$ of size at most $1$ and averaging over $\Aut(q,n)$ leads to the matrix $N$ and conditions (i)--(iii).

The four sets of inequalities in part (iv) come from matrix cut inequalities~\cite{LS}. Let $\yy\in \C^{\emptyset\cup \E}$ be the extended incidence vector of $C$ where $\yy_\emptyset=1$, and $\yy_\uu=1$ if $\uu\in C$ and $\yy_\uu=0$ otherwise for $\uu\in\E$. The vector $\xx=\yy$ satisfies the homogenized cover inequalities
\begin{equation}\label{eq:polytope}
\sum_{i=0}^n \lambda_i\sum_{\vv\in S_i(\uu)} \xx_\vv\geq \beta\xx_{\emptyset} \quad\text{for all $\uu\in\E$}.
\end{equation}
The positive semidefinite matrix $\yy\yy\T=R(M_C)$ has the property that every column satisfies~\eqref{eq:polytope}, and so does the diagonal minus any column. Using the fact that for any $\sigma\in \Aut(q,n)$ also $\sigma C$ has covering radius $r$, and the fact that the sets of vectors $\xx$ satisfying~\eqref{eq:polytope} is closed under nonnegative linear combinations,  we thus obtain two sets of inequalities for (the bordered version of) the matrix $M'$, and similarly for $M''$. This yields the four conditions in (iv).  
\end{remark}

\begin{proof}
To see (i), observe that for every $\sigma$ the number $\sum_{\ell=0}^n\lambda_\ell|\sigma C\cap  S_{\ell}(\zero)|-\beta$ is nonnegative. Since the matrices $R(M_{\sigma C})$ are positive semidefinite, it follows that 
\[
R(c,N)=|\Aut(q,n)|^{-1} \sum_{\sigma \in \Aut(q,n)} R(M_{\sigma C})\cdot\left(\sum_{\ell=0}^n\lambda_\ell |\sigma C\cap  S_{\ell}(\zero)|-\beta\right)
\]
is positive semidefinite.

Item (ii) follows from the fact that for all $\tau\in \AutZ$ and $\ell$ we have
\begin{align*}
\tau\left( \sum_{\sigma\in \Aut(q,n)} M_{\sigma C}\cdot |\sigma C\cap  S_{\ell}(\zero)|\right)
&=\sum_{\sigma\in \Aut(q,n)} M_{\tau\sigma C}\cdot |\sigma C\cap  S_{\ell}(\zero)|\\
&=\sum_{\sigma\in \Aut(q,n)} M_{\tau\sigma C}\cdot |\tau \sigma C\cap \tau  S_{\ell}(\zero)|\\
&=\sum_{\sigma\in \Aut(q,n)} M_{\tau\sigma C}\cdot |\tau \sigma C\cap  S_{\ell}(\zero)|\\
&=\sum_{\rho\in \Aut(q,n)} M_{\rho C}\cdot |\rho C\cap  S_{\ell}(\zero)|,
\end{align*}
where we used that $ S_{\ell}(\zero)$ is fixed under $\AutZ(q,n)$ in the third equation, and we used the substitution $\rho=\tau\sigma$ in the last equality.

To see (iii), let $\uu,\vv\in \E$. We have 
\begin{align*}
\sum_{\sigma \in \Aut(q,n)} (M_{\sigma C})_{\uu,\vv}\cdot |\sigma C \cap  S_{\ell}(\zero)| 
&=\sum_{\ww \in  S_{\ell}(\zero)} |\{\sigma \in \Aut(q,n) \mid  \uu, \vv, \ww \in \sigma C\}| \\
&=\sum_{\ww \in  S_{\ell}(\zero)} |\{\sigma \in \Aut(q,n) \mid \uu-\ww, \vv-\ww, \zero \in \sigma C\}| \\
&= |\Aut(q,n)|\cdot\sum_{\ww \in S_{\ell}(\zero)} M'_{\uu-\ww,\vv-\ww}.
\end{align*}
Summing over $\ell$ and using that 
\[
|\Aut(q,n)|^{-1}\sum_{\sigma\in \Aut(q,n)} (M_{\sigma C})_{\uu,\vv}=M_{\uu,\vv}
\]
shows (iii).

To prove (iv), fix $\uu,\vv\in \E$ and let $\sigma\in \Aut(q,n)$. Since $\sigma C$ has covering radius $r$, we have
\begin{align} \label{eq:sigmaCcovering}
\sum_{\ell=0}^n \lambda_\ell |(\sigma C)\cap S_\ell(\vv)|\geq \beta.
\end{align}
Partition $\Aut(q,n)$ into subsets $\Gamma_1,\ldots ,\Gamma_4$ given by
\begin{align*}
&\Gamma_1=\{\sigma\in \Aut(q,n) \mid \zero\in \sigma C,\ \uu\in \sigma C\}, &&\Gamma_2=\{\sigma\in \Aut(q,n) \mid \zero\in \sigma C,\ \uu\not\in \sigma C\}\\
&\Gamma_3=\{\sigma\in \Aut(q,n) \mid \zero\not\in \sigma C,\ \uu\in \sigma C\}, &&\Gamma_4=\{\sigma\in \Aut(q,n) \mid \zero\not\in \sigma C,\ \uu\not\in \sigma C\}.
\end{align*}
Summing~\eqref{eq:sigmaCcovering} over all $\sigma\in \Gamma_1$ and dividing by $|\Aut(q,n)|$ yields the first inequality since
\[
|\Aut(q,n)|^{-1}\sum_{\sigma\in \Gamma_1} \beta=M'_{\zero,\uu} \beta
\]
and for every $\ell\in \{0,\ldots, n\}$ we have  
\begin{align*}
|\Aut(q,n)|^{-1}\sum_{\sigma\in \Gamma_1} |\sigma C\cap S_\ell(\vv)|&=|\Aut(q,n)|^{-1}\sum_{\ww\in S_\ell(\vv)}|\{\sigma\in \Aut(q,n) \mid \zero,\uu,\ww\in \sigma C\}|\\&=\sum_{\ww\in S_\ell(\vv)} M'_{\uu,\ww}.
\end{align*}
Similarly, summing~\eqref{eq:sigmaCcovering} over all $\sigma\in \Gamma_2$, $\sigma\in \Gamma_3$ and $\sigma\in \Gamma_4$ and dividing by $|\Aut(q,n)|$ yields the second, third and fourth inequality, respectively. Here we use that for any $\mathbf{a}, \mathbf{b}\in \E$ we have 
\begin{align*}
|\Aut(q,n)|^{-1}\cdot |\{\sigma\in \Aut(q,n)\mid \zero\in \sigma C,\ \mathbf{a},\mathbf{b}\in \sigma C\}|&=M'_{\mathbf{a},\mathbf{b}}\\
|\Aut(q,n)|^{-1}\cdot |\{\sigma\in \Aut(q,n)\mid \zero\not\in \sigma C,\ \mathbf{a},\mathbf{b}\in \sigma C\}|&=M''_{\mathbf{a},\mathbf{b}}
\qedhere
\end{align*}
\end{proof}

This leads to the following SDP lower bound on covering codes, which is always at least as tight as the LP bound mentioned in the introduction.\footnote{The LP bound requires nonnegativity of certain expressions involving Krawtchouk polynomials, which is equivalent to positive semidefiniteness of $M=M' +M''$. The LP bound further requires that a small subset of the matrix cut inequalities are satisfied. Thus the LP constraints are implied by the SDP constraints of Theorem \ref{thm:originalbound}.}
\begin{theorem}[Covering lower bound]\label{thm:originalbound}
Suppose that every $C\subseteq \E$ with covering radius $r$ satisfies $(\lambda_0,\ldots, \lambda_n)\beta$. Then
\[
K_q(n,r)^3\geq \min_{M,M',M'',N} q^n\sum_{\uu,\vv\in \E}M'_{\uu,\vv}
\]
where the minimum ranges over real matrices $M, M',M'',N$ that satisfy the conditions in Proposition~\ref{prop:first}, Proposition~\ref{prop:second} and Proposition~\ref{prop:third}.
\end{theorem}
\begin{remark}
The conditions in Proposition~\ref{prop:first}(i) and Proposition~\ref{prop:third} express $M$, $M''$ and $N$ in terms of $M'$, so the minimization can be expressed in terms of the matrix variable $M'$ alone.
\end{remark}

\begin{remark}[Comparison of objective values]
We have
\begin{align}\label{allobjectives}
\sqrt[3]{q^n \sum_{\uu,\vv \in \E} M'_{\uu,\vv} } \geq  \sqrt{ q^n\sum_{\uu \in \E} M'_{\uu,\uu}} \geq q^n M'_{\zero,\zero}.
\end{align}
It follows that~$\sqrt[3]{q^n \sum_{\uu,\vv \in \E} M'_{\uu,\vv} }$ gives the tightest lower bound of the three expressions. 

To see~\eqref{allobjectives}, note the following general fact: for~$c \in \R$ and~$M \in \R^{n \times n}$ with~$m:=\text{diag}(M)$  we have
\begin{align}\label{IJPSD}
\begin{pmatrix} c & m\T \\m & M \end{pmatrix} \succeq 0 \quad \Longrightarrow \quad cM-mm\T \succeq 0
\quad \Longrightarrow \quad
c  \inner{M}{J} \geq \inner{M}{I}^2.
\end{align}
By applying~\eqref{IJPSD} to the matrix~$M$ from the Bose-Mesner algebra (and~$c=1$), we obtain
\begin{align}\label{BMobjectives}
q^n \sum_{\uu \in \E} M'_{\uu,\uu} \geq  (q^n M'_{\zero,\zero})^2.
\end{align}
Since $M'\succeq 0$ and $M'_{\uu,\uu}=M'_{\zero,\uu}$ for all $\uu\in\E$, we have $R(M'_{\zero,\zero},M')\succeq 0$. So by applying~$\eqref{IJPSD}$ we obtain
\begin{align*}
M'_{\zero,\zero}\inner{M'}{J}&\geq \inner{M'}{I}^2, \quad\quad \text{so}
\\(M'_{\zero,\zero})^2 \inner{M'}{J}^2 &\geq \inner{M'}{I}^4. 
\end{align*}
By combining this inequality with~$\eqref{BMobjectives}$ we find
\[
\sum_{\uu \in \E} M'_{\uu,\uu} \inner{M'}{J}^2  \geq q^n (M'_{\zero,\zero})^2 \inner{M'}{J}^2 \geq q^n \inner{M'}{I}^4, 
\]
so 
\[
\inner{M'}{J}^2  \geq q^{n} \inner{M'}{I}^3, \quad \text{implying} \quad 
\big( \sum_{\uu,\vv \in \E} M'_{\uu,\vv}\big)^2 \geq q^n \big(\sum_{\uu \in \E} M'_{\uu,\uu} \big)^3.
\]
Multiplying both sides by~$q^{2n}$
and combining with~\eqref{BMobjectives}  this yields~\eqref{allobjectives}, as desired.
\end{remark}

In the next sections, we make the SDP from Theorem~\ref{thm:originalbound} effective by using symmetry reduction.

\section{Block-diagonalization of the Terwilliger algebra\label{sec:terwilliger}}

In this section, we describe the block-diagonalization of the Terwilliger algebras of the binary~\cite{Sch05} and nonbinary~\cite{GST06} Hamming schemes. This forms the algebraic foundation for the symmetry reduction of the SDP in Theorem~\ref{thm:originalbound}.

\subsection{Terwilliger algebra of the binary Hamming cube \label{sec:terwilligerbinary}}
For $\uu,\vv\in \{0,1\}^n$ we define $\overline{d}(\uu,\vv):=(i,j,t)$, where 
\begin{align*}
i&=|\{\ell\mid \uu_\ell\neq 0\}|,\\
j&=|\{\ell\mid \vv_\ell\neq 0\}|,\\
t&=|\{\ell\mid \uu_\ell\neq 0, \vv_\ell\neq 0\}|.
\end{align*}
The set of tuples $(i,j,t)$ that occur this way are given by \[
I(2,n)=\{(i,j,t)\mid 0\leq t\leq i,j,\ i+j\leq n+t\}
\]
and they parametrize the orbits of pairs $(\uu,\vv)$ under the group $\AutZ(2,n)=\{\sigma\in \Aut(2,n)\mid \sigma(\zero)=\zero\}$ of symmetries of the Hamming space that stabilize the zero word.

For $(i,j,t)\in I(2,n)$ we define the $\{0,1\}^n\times\{0,1\}^n$ matrix $M_{i,j}^{t}$ by 
\[
\left(M_{i,j}^{t}\right)_{\uu,\vv}=
\begin{cases}
1& \text{if }\overline{d}(\uu,\vv)=(i,j,t)\\
0& \text{otherwise.}
\end{cases}
\]
The matrices $M_{i,j}^{t}$ constitute a basis for the space $\mathcal{A}_{2,n}$ of complex $\{0,1\}^n\times \{0,1\}^n$ matrices that are invariant under the action of $\AutZ(2,n)$: 
\[
\mathcal{A}_{2,n}=\left\{\sum_{(i,j,t)\in I(2,n)}x_{i,j}^{t}M_{i,j}^{t}\mid x_{i,j}^{t}\in \C\right\}.
\]
Note that~$(M_{i,j}^t)^{\sf T} = M_{j,i}^t$. Moreover, if~$|\supp(\uu)|=i$, $|\supp(\vv)|=j$, then~$|\supp(\uu) \cap \supp(\vv)|=t$ is equivalent to~$d_H(\uu,\vv)=i+j-2t$. The set~$\mathcal{A}_{2,n}$ is a~$C^*$-algebra, i.e., it is closed under taking the adjoint, addition, scalar multiplication, and matrix multiplication. This algebra is the \emph{Terwilliger algebra of the binary Hamming scheme}. We will now describe the block diagonalization of this algebra, which was derived by Schrijver~\cite{Sch05}, and which allows a significant reduction in the size of matrices of~$\mathcal{A}_{2,n}$ by mapping it to an isomorphic smaller algebra. To this end, define, for $i,j,k,t \in \{0,\ldots, n \}$, the numbers
\begin{align}
\beta_{i,j,k}^{t}=\sum_{u=0}^n (-1)^{t-u} \tbinom{u}{t}\tbinom{n-2k}{u-k}\tbinom{n-k-u}{i-u}\tbinom{n-k-u}{j-u}.
\end{align}
Throughout the paper, we use the convention that $\tbinom{s}{t}=0$ if $s<t$ or $t<0$.
\begin{theorem}{\cite{Sch05}} \label{thm:blockhammingcube}
The following map is a~$*$-isomorphism of algebras: 
\begin{align}
\phi: \mathcal{A}_{2,n} &\to\bigoplus_{k=0}^{\lfloor n/2 \rfloor} \mathbb{C}^{n-2k+1 \times n-2k+1}  \notag \\
\sum_{(i,j,t)\in I(2,n)} x_{i,j}^t M_{i,j}^t &\mapsto \bigoplus_{k=0}^{\lfloor n/2 \rfloor} \left( \sum_{t} \tbinom{n-2k}{i-k}^{-\tfrac{1}{2}} \tbinom{n-2k}{j-k}^{-\tfrac{1}{2}} \beta_{i,j,k}^{t} x_{i,j}^t \right)_{i,j=k}^{n-k}. \label{eq:stariso}
\end{align}
\end{theorem}
Theorem~\ref{thm:blockhammingcube} implies that, for coefficients~$x_{i,j}^t \in \mathbb{C}$, we have
\begin{align}\label{eq:Mprimesymmetryreduction}
\sum_{(i,j,t)\in I(2,n)} x_{i,j}^t M_{i,j}^t \succeq 0  \quad \Longleftrightarrow  \quad \left( \sum_{t} \beta_{i,j,k}^{t} x_{i,j}^t \right)_{i,j=k}^{n-k} \succeq 0 \text{ for all $k=0,\ldots, {\lfloor n/2 \rfloor}$.} 
\end{align}
This crucial fact will be used widely in the SDP-relaxations studied in this paper. Also note that $\phi$ maps real valued matrices to real valued matrices. To see the size of the reduction, note that the matrices in $\mathcal{A}_{2,n}$ have size $2^n \times 2^n$, while (cf.~\cite{Sch05})  \begin{align}| I(2,n)|  = \sum_{k=0}^{\lfloor n/2 \rfloor} (n-2k+1)^2= \binom{n+3}{3},\label{dimA2n}\end{align} which equals the dimension of $\mathcal{A}_{2,n}$.

In the semidefinite relaxations in this paper, we will require not only that elements from~$\mathcal{A}_{2,n}$ are positive semidefinite, but some of the conditions are slightly stronger and inspired by~\cite{Laurent}, namely that matrices of the form 
\[
R(M)=\begin{pmatrix} 1 & \diag(M)^*\\ \diag(M) & M \end{pmatrix} \quad \text{for $M=\sum_{(i,j,t)\in I(2,n)} x_{i,j}^t M_{i,j}^t  \in \mathcal{A}_{2,n}$}, 
\] 
are positive semidefinite. Note that the diagonal of $M$ is given by $\diag M=\sum_{i=0}^n x_{i,i}^i \one_{S_i(\zero)}$. To express positive semidefiniteness of $R(M)$ in terms of $\phi(M)$, observe that for $i,j\in \{0,\ldots, n\}$ the matrix 
\[
\one_{S_i(\zero)}(\one_{S_j(\zero)})\T=\sum_{t\mid (i,j,t)\in I(2,n)} M_{i,j}^t
\]
belongs to $\mathcal{A}_{2,n}$. We have the following lemma. 
\begin{lemma}\label{lem:imageofJ}
Let $A_0\oplus\cdots\oplus A_{\lfloor n/2\rfloor}$ be the image of $\one_{S_i(\zero)}(\one_{S_j(\zero)})\T$ under the map $\phi$. Then 
\[
A_{0}=\tbinom{n}{i}^{\tfrac{1}{2}}\tbinom{n}{j}^{\tfrac{1}{2}}\cdot \big(\delta_{i,i'}\delta_{j,j'}\big)_{i',j'=0}^n,
\]
and $A_{k}$ is the zero matrix for $k>0$.
\end{lemma}

\begin{remark}\label{rem:basisvectors} In~\cite{Sch05}, the $*$-algebra isomorphism $\phi$ is obtained via a block diagonalisation $\phi(M)=\bigoplus_{k=0}^{\lfloor n/2\rfloor}U_k\T MU_k$ for explicit matrices $U_k$, where the columns of the $U_k$ together form an orthonormal system of vectors in $\R^{\{0,1\}^n}$. The columns of $U_0$ are $\binom{n}{i}^{-\tfrac{1}{2}}\one_{S_i(\zero)}$ ($i=0,\ldots, n$). From this, Lemma~\ref{lem:imageofJ} follows directly. Here, we give a proof that only uses the description of $\phi$ as given in Theorem~\ref{thm:blockhammingcube}. 
\end{remark}

\begin{proof}
We first show the following identity:
\begin{equation}\label{eq:betasum}
\sum_{t=0}^n\beta_{i,j,k}^t=
\begin{cases}\tbinom{n}{i}\tbinom{n}{j}&\text{if $k=0$,}\\
0&\text{otherwise.}
\end{cases}
\end{equation}
To see this, we calculate
\begin{align*}
\sum_{t=0}^n \beta_{i,j,k}^{t}&=\sum_{t=0}^n \sum_{u=0}^n (-1)^{t-u} \tbinom{u}{t}\tbinom{n-2k}{u-k}\tbinom{n-k-u}{i-u}\tbinom{n-k-u}{j-u}\\
&=\sum_{u=0}^n \tbinom{n-2k}{u-k}\tbinom{n-k-u}{i-u}\tbinom{n-k-u}{j-u}\cdot \sum_{t=0}^n (-1)^{t-u} \tbinom{u}{t}.
\end{align*}
Since $\sum_{t=0}^n (-1)^{t-u} \tbinom{u}{t}=0$ for $u>0$, we may restrict the summation to $u=0$ and find that 
\[
\sum_{t=0}^n \beta_{i,j,k}^{t}=\tbinom{n-2k}{-k}\tbinom{n-k}{i}\tbinom{n-k}{j}=\delta_{k,0}\tbinom{n}{i}\tbinom{n}{j}.
\]
This proves equation (\ref{eq:betasum}). 

To complete the proof of the lemma, recall that $\one_{S_i(\zero)}(\one_{S_j(\zero)})\T=\sum_{t\mid (i,j,t)\in I(2,n)} M_{i,j}^t$. 
So by Theorem~\ref{thm:blockhammingcube}
\[
A_k= (\sum_{t=0}^n \beta_{i,j,k}^{t})\cdot \tbinom{n-2k}{i-k}^{-\tfrac{1}{2}} \tbinom{n-2k}{j-k}^{-\tfrac{1}{2}} \big(\delta_{i,i'}\delta_{j,j'}\big)_{i',j'=k}^{n-k}.
\]
The result now follows directly from (\ref{eq:betasum}).
\end{proof}

\begin{proposition}\label{prop:RMpsdbinary}
Let $M=\sum_{(i,j,t)\in I(2,n)}x_{i,j}^{t}M_{i,j}^{t}\in \mathcal{A}_{2,n}$ and let 
$\phi(M)=A_0\oplus\cdots\oplus A_{\lfloor n/2\rfloor}$.
Then, for any $c>0$ we have 
\[
R(c,M)\succeq 0\quad\iff\quad A_{k}\succeq 0 \text{ for all $k>0$ and } \begin{pmatrix}c&y^*\\y&A_{0}\end{pmatrix}\succeq 0,
\]
where $y_i=x_{i,i}^{i} \tbinom{n}{i}^{\tfrac{1}{2}}$ for $i=0,\ldots, n$.
\end{proposition}
A proof of this proposition was indicated in~\cite[Lemma 1]{Laurent} by using additional details of the proof of Theorem \ref{thm:blockhammingcube}. Here, we deduce it directly from this theorem.
\begin{proof}
Since $\phi$ is linear, we may assume that $c=1$ (by replacing $M$ with $\tfrac{1}{c}M$.) 
By taking Schur complements, we have 
\begin{align*}
\begin{pmatrix}1&(\diag M)^*\\\diag M&M\end{pmatrix}\succeq 0 &\iff M-\sum_{i,j=0}^n x_{i,i}^{i}(x_{j,j}^{j})^*\one_{S_i(\zero)}(\one_{S_j(\zero)})\T\succeq 0\\
&\iff \phi\left(M-\sum_{i,j=0}^n x_{i,i}^{i}(x_{j,j}^{j})^*\one_{S_i(\zero)}(\one_{S_j(\zero)})\T\right)\succeq 0\\
&\iff A_{k}\succeq 0 \text{ for $k>0$},\ A_{0}-\left( x_{i,i}^{i}(x_{j,j}^{j})^*\tbinom{n}{i}^{\tfrac{1}{2}}\tbinom{n}{j}^{\tfrac{1}{2}}\right)_{i,j=0}^n\succeq 0\\
&\iff A_k\succeq 0 \text{ for $k>0$},\ \begin{pmatrix}1&y^*\\y&A_{0}\end{pmatrix}\succeq 0,
\end{align*}
where we used Theorem~\ref{thm:blockhammingcube} in the second equivalence, Lemma~\ref{lem:imageofJ} in the third equivalence, and Schur's complement in the fourth equivalence.
\end{proof}

\subsection{Terwilliger algebra of the nonbinary Hamming scheme \label{sec:terwilligerqary}}
For $\uu,\vv\in [q]^n$ we define $\overline{d}(\uu,\vv):=(i,j,t,p)$ where 
\begin{align*}
i&=|\{\ell\mid \uu_\ell\neq 0\}|,\\
j&=|\{\ell\mid \vv_\ell\neq 0\}|,\\
t&=|\{\ell\mid \uu_\ell\neq 0, \vv_\ell\neq 0\}|,\\
p&=|\{\ell\mid \uu_\ell=\vv_\ell\neq 0\}|.
\end{align*}
The set of tuples $(i,j,t,p)$ that occur this way are given by \[
I(q,n)=\{(i,j,t,p)\mid 0\leq p\leq t\leq i,j,\ i+j\leq n+t\}
\]
and they parametrize the orbits of pairs $(\uu,\vv)$ under the group $\AutZ(q,n)=\{\sigma\in \Aut(q,n)\mid \sigma(\zero)=\zero\}$ of symmetries of the Hamming space that stabilize the zero word.

For $(i,j,t,p)\in I(q,n)$ we define the $[q]^n\times [q]^n $ matrix $M_{i,j}^{t,p}$ by 
\[
\left(M_{i,j}^{t,p}\right)_{\uu,\vv}=
\begin{cases}
1& \text{if }\overline{d}(\uu,\vv)=(i,j,t,p)\\
0& \text{otherwise.}
\end{cases}
\]
The matrices $M_{i,j}^{t,p}$ constitute a basis for the space $\mathcal{A}_{q,n}$ of complex $[q]^n\times [q]^n$ matrices that are invariant under the action of $\AutZ(q,n)$: 
\[
\mathcal{A}_{q,n}=\left\{\sum_{(i,j,t,p)\in I(q,n)}x_{i,j}^{t,p}M_{i,j}^{t,p}\mid x_{i,j}^{t,p}\in \C\right\}.
\]
This space is a $C^*$-algebra and coincides with the Terwilliger algebra of the nonbinary Hamming scheme. An explicit block diagonalisation of $\mathcal{A}_{q,n}$ was given in \cite{GST06}. To describe it, define for any nonnegative integers $i,j,t,p,a,k$ the number 
\[
\alpha(i,j,t,p,a,k)=\beta^{n-a,t-a}_{i-a,j-a,k-a}\cdot(q-1)^{\tfrac{1}{2}(i+j)-t}\cdot \sum_{g=0}^p(-1)^{a-g}\tbinom{a}{g}\tbinom{t-a}{p-g}(q-2)^{t-a-p+g}.
\]

\begin{theorem}[\cite{GST06}]\label{thm:blocknonbin}
The following map is a~$*$-isomorphism of algebras: 
\begin{align}
\phi: \mathcal{A}_{q,n} &\to \bigoplus_{0\leq a\leq k\leq n+a-k} \mathbb{C}^{(n+a-2k+1) \times (n+a-2k+1)}  \notag \\
\sum_{(i,j,t,p)} x_{i,j}^{t,p} M_{i,j}^{t,p} &\mapsto \bigoplus_{0\leq a\leq k\leq n+a-k} \left( \sum_{t,p} \tbinom{n+a-2k}{i-k}^{-\tfrac{1}{2}} \tbinom{n+a-2k}{j-k}^{-\tfrac{1}{2}} \alpha(i,j,t,p,a,k) x_{i,j}^{t,p} \right)_{i,j=k}^{n+a-k}. \label{eq:starisononbin}
\end{align}
\end{theorem}
This theorem implies that a matrix $M\in \mathcal{A}_{q,n}$ is positive semidefinite if and only if $\phi(A)$ is positive semidefinite. We also note that for real valued matrices $A$, the image $\phi(A)$ is again real. To see the size of the reduction, note  that the matrices in $\mathcal{A}_{q,n}$ have size $q^n \times q^n$, while for~$q\geq 3$ we have (cf.~\cite{GST06}) \begin{align}| I(q,n)|  = \sum_{0\leq a\leq k\leq n+a-k} (n+a-2k+1)^2= \binom{n+4}{4}, \label{dimAqn}\end{align}  which equals the dimension of $\mathcal{A}_{q,n}$.

As in the binary case, we need to deal with bordered matrices $R(M)$ for $M\in \mathcal{A}_{q,n}$. To express semidefiniteness of $R(M)$ in terms of $\phi(M)$, we need the following two lemmas\footnote{Analogous to the binary case (see Remark~\ref{rem:basisvectors}), Lemma~\ref{lem:border} can also be derived directly from the explicit block diagonalisation given in~\cite{GST06}.}.
\begin{lemma}\label{lem:alphas}
We have 
\[
\sum_{p,t}\alpha(i,j,t,p,a,k)=\begin{cases}\tbinom{n}{i}\tbinom{n}{j}(q-1)^{\tfrac{1}{2}(i+j)}&\text{if $a=k=0$}\\0&\text{otherwise}.\end{cases}
\]
\end{lemma}
\begin{proof}
First we observe that by substituting $s=p-g$, we find
\begin{align*}
\sum_{p=0}^t\sum_{g=0}^p(-1)^{a-g}\tbinom{a}{g}\tbinom{t-a}{p-g}(q-2)^{t-a-p+g}& = \sum_{g\geq 0}\sum_{s\geq 0}(-1)^{a-g}\tbinom{a}{g}\tbinom{t-a}{s}(q-2)^{t-a-s} \\
& = \sum_{g\geq 0}(-1)^{a-g}\tbinom{a}{g}\cdot \sum_{s\geq 0}\tbinom{t-a}{s}(q-2)^{t-a-s}.  
\end{align*}
This expression is zero if $a>0$ and is equal to $(q-1)^t$ if $a=0$. 

So we may restrict to the case $a=0$. By (\ref{eq:betasum}) we have
\[
\sum_{t,p}\alpha(i,j,t,p,0,k)=(q-1)^{\tfrac{1}{2}(i+j)}\sum_t\beta_{i,j,k}^{t}=(q-1)^{\tfrac{1}{2}(i+j)}\tbinom{n}{i}\tbinom{n}{j}\delta_{k,0}.
\]
This concludes the proof.
\end{proof}

Similar to the binary case, we have
\[
\one_{S_i(\zero)}(\one_{S_j(\zero)})\T=\sum_{t,p\mid (i,j,t,p)\in I(q,n)}M_{i,j}^{t,p}\in \mathcal{A}_{q,n}.
\]
\begin{lemma}\label{lem:border}
Let $\bigoplus_{0\leq a\leq k\leq n+a-k} A_{a,k}$ be the image of $\one_{S_i(\zero)}(\one_{S_j(\zero)})\T$ under the map $\phi$. Then 
\[
A_{0,0}=(q-1)^{\tfrac{1}{2}(i+j)}\tbinom{n}{i}^{\tfrac{1}{2}}\tbinom{n}{j}^{\tfrac{1}{2}}\big(\delta_{i,i'}\delta_{j,j'}\big)_{i',j'=0}^n,
\]
and $A_{a,k}$ is the zero matrix for $(a,k)\neq (0,0)$.
\end{lemma}
\begin{proof}
Since $\one_{S_i(\zero)}(\one_{S_j(\zero)})\T=\sum_{p,t\mid (i,j,t,p)\in I(q,n)} M_{i,j}^{t,p}$,  
this follows directly from Lemma~\ref{lem:alphas}.
\end{proof}

\begin{proposition}\label{prop:RMpsdnonbin}
Let $M=\sum_{(i,j,t,p)\in I(q,n)}x_{i,j}^{t,p}M_{i,j}^{t,p}\in \mathcal{A}_{q,n}$ and let 
\[
\phi(M)=\bigoplus_{0\leq a\leq k\leq n+a-k} A_{a,k}.
\]
Then $R(M)$ is positive semidefinite if and only if $A_{a,k}\succeq 0$ for all $(a,k)\neq (0,0)$ and 
\[
\begin{pmatrix}1&y^*\\y&A_{0,0}\end{pmatrix}\succeq 0,
\]
where $y_i= x_{i,i}^{i,i} (q-1)^{i/2}\tbinom{n}{i}^{\tfrac{1}{2}}$.
\end{proposition}
\begin{proof}
By taking Schur complements, using Theorem~\ref{thm:blocknonbin} and using Lemma~\ref{lem:border} we find 
\begin{alignat*}{3}
R(M)\succeq 0 &\iff&&\quad M-\sum_{i,j=0}^nx_{i,i}^{i,i}(x_{j,j}^{j,j})^*\one_{S_i(\zero)}(\one_{S_j(\zero)})\T\succeq 0\\
&\iff&&\quad \phi\left(M-\sum_{i,j=0}^nx_{i,i}^{i,i}(x_{j,j}^{j,j})^*\one_{S_i(\zero)}(\one_{S_j(\zero)})\T \right)\succeq 0\\
&\iff&&\quad A_{a,k}\succeq 0 \text{ for $(a,k)\neq (0,0)$,}\\
&&&\quad A_{0,0}-\sum_{i,j=0}^nx_{i,i}^{i,i}(x_{j,j}^{j,j})^*(q-1)^{\tfrac{1}{2}(i+j)}\tbinom{n}{i}^{\tfrac{1}{2}}\tbinom{n}{j}^{\tfrac{1}{2}}\big(\delta_{i,i'}\delta_{j,j'}\big)_{i',j'=0}^n\succeq 0.
\end{alignat*}
The result now follows by taking Schur complements.
\end{proof}

\section{The symmetry-reduced SDP bound}
In this section we demonstrate how the block-diagonalization from the Terwilliger algebra as described in Section~\ref{sec:terwilliger} can be employed to significantly reduce the new SDP bound for covering codes from Section~\ref{sec:generalsdp}.

\subsection{The SDP bound for binary covering codes\label{sec:binary}}

 In the binary case, the matrices~$M'$,~$M''$ and~$N$ from Section~\ref{sec:generalsdp} are contained in the algebra~$\mathcal{A}_{2,n}$. So we can write
\begin{align}\label{eq:Mprimebinary}
M'= \sum_{(i,j,t) \in I(2,n)} x_{i,j}^t M_{i,j}^t 
\end{align}
for real numbers $x_{i,j}^t$.
\begin{lemma}\label{lem:matricesbinary}
We have 
\[
M=\sum_{(i,j,t)\in I(2,n)} x_{i+j-2t,0}^0 M_{i,j}^t,\qquad M'' =  \sum_{(i,j,t)\in I(2,n)} (x_{i+j-2t,0}^0 - x_{i,j}^t) M_{i,j}^t.
\]
\end{lemma}
\begin{proof}
Note that if $\uu,\vv\in \E$ with $\overline{d}(\uu,\vv)=(i,j,t)$, then $d(\uu,\vv)=i+j-2t$. So $M_{\uu,\vv} = x_{i+j-2t,0}^0$ and $M''_{\uu,\vv} = x_{i+j-2t,0}^0 - x_{i,j}^t$ by Proposition~\ref{prop:first} (i). 
\end{proof}

\begin{remark}[Interpretation of the $x_{i,j}^t$]
The coefficients $x_{i,j}^{t}$ provide insight into the structure of the code~$C$ by extending the concept of distance distribution in Delsarte's linear programming approach to triples. Whereas the distance distribution tracks the number of pairs in $C$ at each distance $d$, the coefficients $x_{i,j}^{t}$ count the number of \emph{triples} $(\uu, \vv, \ww) \in C^3$ that belong to an equivalence class of $\E^3$ under the group action of $\text{Aut}(2,n)$. We proceed by explaining this formally. 
Define
\[
X_{i,j,t} := \{ (\uu, \vv, \ww) \in \E \times \E \times \E \mid \overline{d}( \vv-\uu, \ww-\uu) = (i, j, t) \},
\]
for  $(i, j, t) \in I(2, n) $. For each $(i, j, t) \in I(2, n)$, define the numbers
\[
\lambda_{i,j}^{t} := |(C \times C \times C) \cap X_{i,j,t}|,
\]
and let
\[
\gamma_{i,j}^{t} := |(\{\mathbf{0}\} \times \E \times \E) \cap X_{i,j,t}| = \binom{n}{ i-t, j-t,t}
\]
be the number of nonzero entries of $ M_{i,j}^t$. Then the coefficients $x_{i,j}^{t}$ are given by
\begin{align}\label{eq:coefficientsxijt}
x_{i,j}^{t} = 2^{-n} (\gamma_{i,j}^{t})^{-1} \lambda_{i,j}^{t}.
\end{align}
To see~\eqref{eq:coefficientsxijt}, notice that the matrices $ M_{i,j}^t$ are pairwise orthogonal, and their inner products satisfy $\langle M_{i,j}^t, M_{i,j}^t \rangle = \gamma_{i,j}^{t}$, for $(i, j, t) \in I(2, n)$. Thus, we compute:
\begin{align*}
\langle M', M_{i,j}^t \rangle &= |\text{Aut}(2,n)|^{-1} \sum_{\uu \in C} \sum_{\substack{\sigma \in \text{Aut}(2,n) \\ \sigma \uu = \mathbf{0}}} \langle M_{\sigma C}, M_{i,j}^t \rangle \\
&= |\text{Aut}(2,n)|^{-1} \cdot |\text{Aut}_{\mathbf{0}}(2,n)| \sum_{\uu \in C}  |(\{\uu\} \times C \times C) \cap X_{i,j,t}| \\
&= 2^{-n} |(C \times C \times C) \cap X_{i,j,t}| = 2^{-n} \lambda_{i,j}^{t}.
\end{align*}
Hence we have
\[
M' = 2^{-n} \sum_{(i,j,t) \in I(2,n)} \lambda_{i,j}^{t} (\gamma_{i,j}^{t})^{-1} M_{i,j}^t.
\]
Comparing the coefficients of each $ M_{i,j}^t$ in this expression with those in~\eqref{eq:Mprimebinary} establishes~\eqref{eq:coefficientsxijt}. 
\end{remark}

\begin{proposition}[Basic inequalities and symmetry]\label{prop:firstbinary}
    The $x_{i,j}^t$ satisfy 
\begin{align}\label{constraints:linearbinary}
\text{\emph{(i)}}& \quad\quad 0 \leq x_{i,j}^t \leq x_{i,i}^i, \\ 
\text{\emph{(ii)}}& \quad\quad x_{i,0}^0 + x^0_{i+j-2t,0}-x_{0,0}^0  \leq x_{i,j}^t\leq    x^0_{i+j-2t,0},\notag \\ 
\text{\emph{(iii)}}&\quad\quad  x_{i,j}^t = x_{i',j'}^{t'} \,\,\,\text{ if $(i,j,i+j-2t)$ is a permutation of $(i',j',i'+j'-2t')$},\notag
\end{align}
\end{proposition}
\begin{proof}
This follows directly from Proposition~\ref{prop:first} (ii), (iii), (iv) combined with Lemma~\ref{lem:matricesbinary}.
\end{proof} 

We now employ the block-diagonalization of~$\mathcal{A}_{2,n}$ from Section~\ref{sec:terwilligerbinary} to the matrices of Proposition~\ref{prop:second}.

\begin{proposition}[Semidefiniteness]\label{prop:secondbinary}
    The following matrices are positive semidefinite:
\begin{align}\label{eq:terwilligerbinarypsd}
\left(\sum_{t=0}^n \beta_{i,j,k}^t x_{i,j}^t  \right)_{i,j=k}^{n-k} \succeq 0, \quad &\left(\sum_{t=0}^n \beta_{i,j,k}^t (x^0_{i+j-2t,0}-x_{i,j}^t)  \right)_{i,j=k}^{n-k} \succeq 0, \quad \text{for $k=1,\ldots, \lfloor\tfrac{n}{2}\rfloor$},\\
\left(\sum_{t=0}^n \beta_{i,j,0}^t x_{i,j}^t  \right)_{i,j=0}^{n} \succeq 0,\quad &\begin{pmatrix} 1-x_{0,0}^0 & y\T \\ y & L \end{pmatrix} \succeq 0, \notag
\end{align}
where $L= \left(\sum_{t=0}^n \beta_{i,j,0}^t (x^0_{i+j-2t,0}-x_{i,j}^t)  \right)_{i,j=0}^{n}$, and~$y_i:=\tbinom{n}{i}(x_{0,0}^0-x_{i,0}^0)$, \text{ for $i=0,\ldots,n$}. 
\end{proposition}
\begin{proof} 
This follows from applying the block-diagonalization from Section~\ref{sec:terwilligerbinary} to the matrices in Proposition~\ref{prop:second}. Here Equation~\eqref{eq:Mprimesymmetryreduction} directly gives the reduction of the matrix~$M'$, and  Proposition~\ref{prop:RMpsdbinary} combined with Equation~\eqref{eq:Mprimesymmetryreduction} gives the reduction of the matrix $R(1-x_{0,0}^0,M'')$.
\end{proof}

Next, we will describe the symmetry reduction of the constraints from Proposition~\ref{prop:third}. We start with the constraints following from the matrix~$N$ from~\eqref{eq:Nmatrix}. Recall that the matrix~$N$ satisfies (i), (ii), (iii) from Proposition~\ref{prop:third}. We aim to block-diagonalize the matrix given in (i) of Proposition~\ref{prop:third}. To this end, we first define the following numbers. Given two words~$\uu,\vv \in \mathbb{E}$, with~$\overline{d}(\uu,\vv)=(i,j,t)$, we denote by $\eta_{(i',j',t'),d}^{(i,j,t)}$ the number of words~$\ww \in \mathbb{E}$ with weight~$d$, and~$\overline{d}(\uu-\ww,\vv-\ww)=(i',j',t')$. Note that necessarily $i+j-2t=i'+j'-2t'$. 
\begin{lemma}\label{lem:etabinary}
 The numbers~$\eta_{(i',j',t'),d}^{(i,j,t)}$ satisfy
\begin{align}
\eta_{(i',j',t'),d}^{(i,j,t)} =  \sum_{a_{00},a_{01},a_{10},a_{11}} \binom{i-t}{a_{10}} \binom{j-t}{a_{01}} \binom{t}{a_{11}}\binom{n+t-i-j}{a_{00}},
\end{align}
where the indices~$a_{00}$, $a_{01}$, $a_{10}$ and~$a_{11}$ range over all nonnegative integers for which
\begin{align}
a_{00} &\leq n+t-i-j, \quad a_{01} \leq j-t, \quad a_{10} \leq i-t, \quad a_{11} \leq t, \\ 
  d&= a_{00}+a_{01}+a_{10}+a_{11},\notag\\
i'&= i+a_{00}-a_{11}-a_{10}+a_{01},\notag\\
j'&= j+a_{00}-a_{11}+a_{10}-a_{01}.\notag
\end{align}
\end{lemma}
\begin{proof}
Partition the support of $\ww$ into sets $A_{00}$, $A_{01}$, $A_{10}$ and $A_{11}$, where for $\alpha,\beta\in \{0,1\}$ the set $A_{\alpha\beta}$ is given by
\[
A_{\alpha\beta}=\{\ell\in\supp(\ww )\mid u_\ell=\alpha,\ v_\ell=\beta\}.
\]
Denoting $a_{\alpha\beta}=|A_{\alpha\beta}|$, the result follows by summing over all possible subsets $A_{00},A_{01},A_{10},A_{11}$.  
\end{proof}

Applying the block-diagonalization of~$\mathcal{A}_{2,n}$ from Section~\ref{sec:terwilligerbinary} to the matrix (i) of Proposition~\ref{prop:third}, i.e., to the matrix $R(c,N)$, where $c=\sum_{\ell=0}^n \lambda_{\ell}\cdot|S_{\ell}(\zero)|\cdot M'_{\zero,\zero}
 - \beta$, gives the following result.

\begin{proposition}[Lasserre constraint]\label{prop:thirdlasserrebinary}
If~$C$ satisfies the inequality $(\lambda_0,\ldots,\lambda_n)\beta$, then the 
   following matrices are positive semidefinite:
\begin{align}\label{constraints:lasserre}
&\left(\sum_{t=0}^n \beta_{i,j,k}^t \left(\left(\sum_{d=0}^n\sum_{i',j',t'}\lambda_d \eta_{(i',j',t'),d}^{(i,j,t)} x_{i',j' }^{t'}\right) - \beta x^0_{i+j-2t,0} \right)\right)_{i,j=k}^{n-k} \succeq 0 \, & \quad \text{for each $k=1,\ldots, \lfloor\tfrac{n}{2}\rfloor$}, \\
&\begin{pmatrix} \sum_{i=0}^n \binom{n}{i}\lambda_i x_{0,0}^0 -\beta & y\T \\ y & T \end{pmatrix} \succeq 0, \notag
\end{align}
where 
\begin{align*}
    T=\left(\sum_{t=0}^n \beta_{i,j,0}^t \left(\left(\sum_{d=0}^n\sum_{i',j',t'}\lambda_d \eta_{(i',j',t'),d}^{(i,j,t)} x_{i',j' }^{t'}\right) - \beta x^0_{i+j-2t,0} \right)\right)_{i,j=0}^{n}, \\\text{and } y_i:=\binom{n}{i}\left(\left(\sum_{d=0}^n\sum_{i',j',t'}\lambda_d \eta_{(i',j',t'),d}^{(i,i,i)} x_{i',j' }^{t'}\right) - \beta x^0_{0,0} \right), \quad \text{for $i=0,\ldots,n$}.
\end{align*}
\end{proposition}
\begin{proof} 
By Proposition~\ref{prop:third} (i) the matrix $R(c,N)$ is positive semidefinite, where $c=\sum_{\ell=0}^n \lambda_{\ell}\cdot|S_{\ell}(\zero)|\cdot M'_{\zero,\zero} -\beta = \sum_{i=0}^n \binom{n}{i}\lambda_i x_{0,0}^0 -\beta$. We next compute the entries of~$N$ using Proposition~\ref{prop:third} (iii). We have, for~$\uu,\vv \in \{0,1\}^n$, that
\begin{align*}
N_{\uu,\vv}&= -\beta M_{\uu,\vv} + \sum_{d =0}^n \lambda_{d} \sum_{\ww \in S_d(\zero)} M'_{\uu-\ww, \vv-\ww}
\\&= -\beta x_{i+j-2t,0}^0 +  \sum_{d=0}^n \lambda_d\sum_{\ww \in S_d(\bm 0  )} \left(\sum_{i',j',t'} x_{i',j'}^{t'} M_{i',j'}^{t'}  \right)_{\uu-\ww,\vv-\ww} 
 \\&= -\beta x_{i+j-2t,0}^0+ \sum_{d=0}^n \lambda_{d} \sum_{\ww \in S_d(\bm 0 )} \sum_{i',j',t'} x_{i',j'}^{t'}\one_{\overline{d}(\uu-\ww,\vv-\ww)=(i',j',t')} 
 \\&= -\beta x_{i+j-2t,0}^0+ \sum_{d=0}^n \lambda_{d} \sum_{i',j',t'} \left( \sum_{\ww \in S_d(\bm 0 )}  \one_{\overline{d}(\uu-\ww,\vv-\ww)=(i',j',t')} \right) x_{i',j'}^{t'}
 \\&= -\beta x_{i+j-2t,0}^0 + \sum_{d=0}^n\lambda_d \sum_{i',j',t'}\eta_{(i',j',t'),d}^{(i,j,t)} x_{i',j' }^{t'}. 
\end{align*}
Then, Proposition~\ref{prop:RMpsdbinary} together with Equation~\eqref{eq:Mprimesymmetryreduction} yields the block-diagonalization of~$R(c,N)$ as given in~\eqref{constraints:lasserre}.
\end{proof}

It remains to translate the linear constraints from Proposition~\ref{prop:third}~(iv) in terms of the~$x_{i,j}^t$. To do this, we need the following numbers.  Given two words~$\uu$, $\vv \in \E$ with~$\overline{d}(\uu,\vv) = (i,j,t)$, let~$\alpha_{(i,j',t'),d}^{(i,j,t)}$ be the number of words~$\ww \in \E$ with $\overline{d}(\uu,\ww)=(i,j',t')$ and~$d(\vv,\ww) =d$. Note that
 \[
 \alpha_{(i,j',t'),d}^{(i,j,t)}=\eta_{(i+j'-2t',d,t-t'+(j'-j)/2),j'}^{(i,j,t)}
 \]
 (both count the number of extensions of $(\zero,\uu,\vv)$ to a configuration $(\zero,\uu,\vv,\ww)$ in a specified orbit of 4-tuples). 
\begin{lemma}\label{lem:alpha}
 The numbers~$\alpha_{(i,j',t'),d}^{(i,j,t)}$ satisfy
\begin{align}
\alpha_{(i,j',t'),d}^{(i,j,t)} =  \sum_{a_{00},a_{01},a_{10},a_{11}} \binom{i-t}{a_{10}} \binom{j-t}{a_{01}} \binom{t}{a_{11}}\binom{n+t-i-j}{a_{00}},
\end{align}
where the indices~$a_{00}$, $a_{01}$, $a_{10}$ and~$a_{11}$ range over all nonnegative integers for which
\begin{align}
a_{00} &\leq n+t-i-j, \quad a_{01} \leq j-t, \quad a_{10} \leq i-t, \quad a_{11} \leq t,\\ 
j' &= a_{00} + a_{01} + a_{10} + a_{11}\notag\\
t'&= a_{10} + a_{11} \notag\\
d-j &=  a_{00} + a_{10} - a_{01} - a_{11}.\notag 
\end{align}
\end{lemma}
\begin{proof}
Analogous to Lemma~\ref{lem:etabinary}. 
\end{proof}

\begin{lemma}\label{lem:lambdas}
Let $\uu \in \E$ with $d(\uu,\zero)=i$ and let $x \in \mathbb{R}^{\E}$ be such that $x_{\vv}$ is determined by $\overline{d}(\uu ,\vv)$, say $x_{\vv}=x_{i,j}^{t}$, in case $\overline{d}(\uu ,\vv)=(i,j,t)$. Then 
\[
\sum_{d=0}^n \lambda_d x(S_d( \vv ))\geq \beta \quad \text{for all $\vv \in\E$} \quad \Longleftrightarrow \quad
\sum_{j',t'} x_{i,j'}^{t'}\cdot \sum_{d=0}^n \lambda_d \alpha_{(i,j',t'),d}^{(i,j,t)} \geq \beta \quad \text{for all $j,t$},
\]
where $x(S_d( \vv )) := \sum_{\ww \in S_d(\vv)} x_{\ww}$.
\end{lemma}
\begin{proof}
Let $\vv \in \E$ and $\overline{d}(\uu,\vv) = (i,j,t)$. Then we have 
\begin{align*}
\sum_{d=0}^n \lambda_d x(S_d(\vv)) & = \sum_{d=0}^n \lambda_d \sum_{\substack{\ww \in\E \\ d(\vv,\ww) = d}} x_{\ww} = \sum_{d=0}^n \lambda_d \sum_{j',t'} \sum_{\substack{\ww \in \E\\d(\vv, \ww)=d \\ d(\uu, \ww)=(i,j',t')}}x_{\ww} \\
&= \sum_{d=0}^n \lambda_d\sum_{j',t'} \alpha_{(i,j',t'),d}^{(i,j,t)} x_{i,j'}^{t'}= \sum_{j',t'} x_{i,j'}^{t'} \cdot \sum_{d=0}^n \lambda_d \alpha_{(i,j',t'),d}^{(i,j,t)}. \qedhere
\end{align*}
\end{proof}

Now we can rewrite the linear constraints from Proposition~\ref{prop:third} (iv) in terms of the~$x_{i,j}^t$. 

\begin{proposition}[Matrix cut inequalities]\label{prop:thirdmatrixcutbinary} If~$C$ satisfies the inequality $(\lambda_0,\ldots,\lambda_n)\beta$, then for every tuple~$(i,j,t) \in I(2,n)$ the following inequalities hold:
\begin{align} \label{constraints:matrixcut}
    \sum_{j',t'} x_{i,j'}^{t'} \cdot \lambda_{j',t'}^{i,j,t} & \geq x_{i,0}^0 \beta, \\
    \sum_{j',t'} (x_{j',0}^0 - x_{i,j'}^{t'}) \cdot \lambda_{j',t'}^{i,j,t} & \geq (x_{0,0}^0 -  x_{i,0}^0) \beta, \notag\\
    \sum_{j',t'} (x_{i+j'-2t',0}^0 - x_{i,j'}^{t'} )\cdot \lambda_{j',t'}^{i,j,t} & \geq (x_{0,0}^0 -  x_{i,0}^0) \beta,\notag\\
    \sum_{j',t'} (x_{0,0}^0 - x_{j',0}^0 -x_{i+j'-2t',0}^0 + x_{i,j'}^{t'} )\cdot \lambda_{j',t'}^{i,j,t} & \geq (1-2x_{0,0}^0 +  x_{i,0}^0) \beta, \notag
\end{align}
where the numbers~$\lambda_{j',t'}^{i,j,t}$ are given by
\[
\lambda_{j',t'}^{i,j,t} := \sum_{d=0}^n \lambda_d \alpha_{(i,j',t'),d}^{(i,j,t)}.
\]
\end{proposition}
\begin{proof} 
These are the constraints from Proposition~\ref{prop:third} (iv) written in terms of the~$x_{i,j}^t$, making use of Lemma~\ref{lem:lambdas}. 
\end{proof}

The following theorem makes the semidefinite programming lower bound from Theorem~\ref{thm:originalbound} effective in case~$q=2$, by using the results of Propositions~\ref{prop:firstbinary},~\ref{prop:secondbinary},~\ref{prop:thirdlasserrebinary}, and~\ref{prop:thirdmatrixcutbinary}.

\begin{theorem} \label{thm:sdplowerbound}
If every code~$C \subseteq \mathbb{E}$ with covering radius~$r$ satisfies $(\lambda_0,\lambda_1,\ldots,\lambda_n)\beta$, then we have
\begin{align}\label{eq:objectivebinary}
    K_2(n,r)^3 \geq \min_x 2^n\sum_{(i,j,t) \in I(2,n)} \binom{n}{i-t,j-t,t}x_{i,j}^{t}, 
\end{align}
where the~$x_{i,j}^t$ are real numbers that satisfy the constraints in Propositions~\ref{prop:firstbinary},~\ref{prop:secondbinary},~\ref{prop:thirdlasserrebinary}, and \ref{prop:thirdmatrixcutbinary}. 
\end{theorem}
\begin{proof} 
Equation~$\eqref{eq:objectivebinary}$ is the objective function from Theorem~\ref{thm:originalbound} written in terms of the~$x_{i,j}^t$. The validity of the constraints from Propositions~\ref{prop:firstbinary},~\ref{prop:secondbinary},~\ref{prop:thirdlasserrebinary}, and~\ref{prop:thirdmatrixcutbinary} was already proven.   
\end{proof} 

By~\eqref{dimA2n}, both the number of variables of the reduced SDP and the sum of squares of its block sizes are $O(n^3)$. 

\subsection{The SDP bound for nonbinary covering codes\label{sec:nonbinary}}

In this section, we address the nonbinary case, that is, $q \geq 3$. The overall strategy is the same as in the binary case discussed in the previous section, so the arguments are analogous. However, the formulas involve additional terms and variables, making them more cumbersome to write out explicitly. To avoid redundancy, we keep proofs brief.

The matrices~$M'$,~$M''$ and~$N$ from Section~\ref{sec:generalsdp} are contained in the algebra~$\mathcal{A}_{q,n}$. So we can write
\begin{align}\label{eq:Mprimeqary}
M'= \sum_{(i,j,t,p) \in I(q,n)} x_{i,j}^{t,p} M_{i,j}^{t,p} 
\end{align}
for real numbers $x_{i,j}^{t,p}$.
\begin{lemma}\label{lem:matricesqary}
We have 
\[
M=\sum_{i,j,t,p} x_{i+j-t-p,0}^{0,0} M_{i,j}^{t,p},\qquad M'' =  \sum_{i,j,t,p} (x_{i+j-t-p,0}^{0,0} - x_{i,j}^{t,p}) M_{i,j}^{t,p}.
\]
\end{lemma}
\begin{proof}
Note that if $\uu,\vv\in \E$ with $\overline{d}(\uu,\vv)=(i,j,t,p)$, then $d(\uu,\vv)=i+j-t-p$. So $M_{\uu,\vv} = x_{i+j-t-p,0}^{0,0}$ and $M''_{\uu,\vv} = x_{i+j-t-p,0}^{0,0} - x_{i,j}^{t,p}$ by Proposition~\ref{prop:first} (i). 
\end{proof}

\begin{remark}[Interpretation of the $x_{i,j}^{t,p}$]
Similar to the binary case, the coefficients $x_{i,j}^{t,p}$ provide insight into the structure of the code~$C$ by extending the concept of distance distribution in Delsarte's linear programming approach to triples.  The $x_{i,j}^{t,p}$ count the number of \emph{triples} $(\uu, \vv, \ww) \in C^3$ that belong to an equivalence class of $\E^3$ under the group action of $\text{Aut}(q,n)$. 
Define
\[
X_{i,j,t,p} := \{ (\uu, \vv, \ww) \in \E \times \E \times \E \mid \overline{d}( \vv-\uu, \ww-\uu) = (i, j, t,p) \},
\]
for  $(i, j, t,p) \in I(q, n) $. For each $(i, j, t,p) \in I(q, n)$, define the numbers
\[
\lambda_{i,j}^{t,p} := |(C \times C \times C) \cap X_{i,j,t,p}|,
\]
and let
\[
\gamma_{i,j}^{t,p} := |(\{\mathbf{0}\} \times \E \times \E) \cap X_{i,j,t,p}| =(q-1)^{i+j-t} (q-2)^{t-p} \binom{n}{p,t-p, i-t, j-t}
\]
be the number of nonzero entries of $ M_{i,j}^{t,p}$. Then the coefficients $x_{i,j}^{t,p}$ are given by
\begin{align}\label{eq:coefficientsxijtp}
x_{i,j}^{t,p} = q^{-n} (\gamma_{i,j}^{t,p})^{-1} \lambda_{i,j}^{t,p}.
\end{align}
Equation~\eqref{eq:coefficientsxijtp} can be seen with an analogous computation as was done to show~\eqref{eq:coefficientsxijt} for the binary case. To prevent repetition, we omit the details.
\end{remark}

\begin{proposition}[Basic inequalities and symmetry]\label{prop:firstqary}
    The $x_{i,j}^{t,p}$ satisfy 
\begin{align}\label{constraints:qarylinear}
\text{\emph{(i)}}& \quad\quad 0 \leq x_{i,j}^{t,p} \leq x_{i,i}^{i,i}, \\ 
\text{\emph{(ii)}}& \quad\quad x_{i,0}^{0,0} + x^{0,0}_{i+j-t-p,0}-x_{0,0}^{0,0}  \leq x_{i,j}^{t,p}\leq    x^{0,0}_{i+j-t-p,0},\notag \\ 
\text{\emph{(iii)}}&\quad\quad  x_{i,j}^{t,p} = x_{i',j'}^{t',p'} \,\,\,\text{ if $(i,j,i+j-t-p)$ is a permutation of $(i',j',i'+j'-t'-p')$ }\notag\\
& \quad\quad\quad\quad\quad\quad\quad\quad\text{and $t-p=t'-p'$}, \notag 
\end{align}
\end{proposition}
\begin{proof} 
This follows directly from Proposition~\ref{prop:first} (ii), (iii), (iv) combined with Lemma~\ref{lem:matricesqary}.
\end{proof}

We now employ the block-diagonalization of~$\mathcal{A}_{q,n}$ from Section~\ref{sec:terwilligerqary} to the matrices of Proposition~\ref{prop:second}.

\begin{proposition}[Semidefiniteness]\label{prop:secondqary}
For all $a,k$ with~$0 \leq a \leq k \leq n+a-k$ and $k \neq 0$ we have
\begin{equation}\label{constraints:qaryterwilliger}
\left(\sum_{t,p=0}^n \alpha(i,j,t,p,a,k) x_{i,j}^{t,p}  \right)_{i,j=k}^{n+a-k} \succeq 0, \quad \left(\sum_{t,p=0}^n \alpha(i,j,t,p,a,k)(x^{0,0}_{i+j-t-p,0}-x_{i,j}^{t,p})  \right)_{i,j=k}^{n+a-k} \succeq 0.
\end{equation}
Moreover, 
\begin{equation}
\left(\sum_{t,p=0}^n \alpha(i,j,t,p,0,0) x_{i,j}^{t,p}  \right)_{i,j=0}^{n} \succeq 0, \quad \begin{pmatrix}1-x_{0,0}^{0,0} & y\T \\ y & L \end{pmatrix} \succeq 0,
\end{equation}
where 
\begin{equation*}
L:= \left(\sum_{t,p=0}^n \alpha(i,j,t,p,0,0)(x^{0,0}_{i+j-t-p,0}-x_{i,j}^{t,p})\right)_{i,j=0}^{n}\quad \text{and}\quad y_i:= (x_{0,0}^{0,0} - x_{i,i}^{i,i})\binom{n}{i} (q-1)^{\tfrac{1}{2}i}.
\end{equation*}
\end{proposition}
\begin{proof} 
This follows from applying the block-diagonalization from Section~\ref{sec:terwilligerqary} to the matrices in Proposition~\ref{prop:second}. Here Theorem~\ref{thm:blocknonbin} directly gives the reduction of the matrix~$M'$, and  Proposition~\ref{prop:RMpsdnonbin} combined with  Theorem~\ref{thm:blocknonbin} gives the reduction of the matrix $R(1-x_{0,0}^{0,0},M'')$.
\end{proof} 

Analogous to the binary case, we next describe the symmetry reduction of the constraints from Proposition~\ref{prop:third}. We start with the constraints following from the matrix~$N$ from~\eqref{eq:Nmatrix}. Recall that the matrix~$N$ satisfies (i), (ii), (iii) from Proposition~\ref{prop:third}. We aim to block-diagonalize the matrix given in (i) of Proposition~\ref{prop:third}. To do this, we define the following numbers.

Given two words~$\uu,\vv \in \mathbb{E}$, with~$\overline{d}(\uu,\vv)=(i,j,t,p)$, we denote by $\eta_{(i',j',t',p'),d}^{(i,j,t,p)}$ the number of words~$\ww \in \mathbb{E}$ with weight~$d$, and~$\overline{d}(\uu-\ww,\vv-\ww)=(i',j',t',p')$. Note that $i+j-t-p = i' + j' -t'-p'$. As in the binary case, the numbers~$\eta_{(i',j',t',p'),d}^{(i,j,t,p)}$ can be explicitly calculated as a sum of expressions including binomial coefficients.

\begin{lemma}\label{lem:etaqary}
 The numbers~$\eta_{(i',j',t',p'),d}^{(i,j,t,p)}$ satisfy
\begin{align}
\eta_{(i',j',t',p'),d}^{(i,j,t,p)} =  \sum \binom{i-t}{a_1,a_2}&\binom{j-t}{b_1,b_2}\binom{p}{c_1,c_2}\binom{t-p}{d_1,d_2,d_3}\binom{n+t-i-j}{e}\notag\\&\cdot(q-1)^e (q-2)^{a_2+b_2+c_2}(q-3)^{d_3},
\end{align}
where the summation is over all nonnegative integers~$a_1$, $a_2$, $b_1$, $b_2$, $c_1$, $c_2$, $d_1$, $d_2$, $d_3$, and $e$ for which
\begin{align}
e &\leq n+t-i-j, \quad b_1+b_2 \leq j-t,\quad a_1+a_2 \leq i-t,\quad d_1+d_2+d_3 \leq t-p,\quad c_1+c_2 \leq p\\
d &= a_1 + a_2 +b_1 +b_2 +c_1 +c_2 +d_1 +d_2 +d_3 + e \notag \\
i'&= i - a_1 + b_1 + b_2 - c_1 - d_1 + e \notag \\
j'&= j + a_1 + a_2 - b_1 - c_1 - d_2 + e \notag \\
t' &= t +a_2 + b_2 -c_1 - d_1 - d_2 + e. \notag 
\end{align}
\end{lemma}
\begin{proof}
The equality can be seen by partitioning the support of~$\ww$ into sets $A_1$, $A_2$, $B_1$, $B_2$, $C_1$, $C_2$, $D_1$, $D_2$, $D_3$, and $E$ as follows:
\begin{align}
A_1 &:= \{k \in \supp(\ww) \,\,|\,\, \uu_k \neq 0, \vv_k = 0, \ww_k = \uu_k\}, \\
A_2 &:= \{k \in \supp(\ww) \,\,|\,\, \uu_k \neq 0, \vv_k = 0, \ww_k \neq \uu_k\}, \notag \\
B_1 &:= \{k \in \supp(\ww) \,\,|\,\, \uu_k = 0, \vv_k \neq 0, \ww_k = \vv_k\}, \notag \\
B_2 &:= \{k \in \supp(\ww) \,\,|\,\, \uu_k = 0, \vv_k \neq 0, \ww_k \neq \vv_k\}, \notag \\
C_1 &:= \{k \in \supp(\ww) \,\,|\,\, \uu_k \neq 0, \vv_k = \uu_k, \ww_k = \uu_k\}, \notag \\
C_2 &:= \{k \in \supp(\ww) \,\,|\,\, \uu_k \neq 0, \vv_k = \uu_k, \ww_k \neq \uu_k\}, \notag \\
D_1 &:= \{k \in \supp(\ww) \,\,|\,\, \uu_k \neq 0, \vv_k \neq 0, \ww_k = \uu_k\}, \notag \\
D_2 &:= \{k \in \supp(\ww) \,\,|\,\, \uu_k \neq 0, \vv_k \neq 0, \ww_k = \vv_k\}, \notag \\
D_3 &:= \{k \in \supp(\ww) \,\,|\,\, \uu_k \neq 0, \vv_k \neq 0, \ww_k \neq \uu_k, \ww_k \neq \vv_k\}, \notag \\
E &:= \{k \in \supp(\ww) \,\,|\,\, \uu_k = 0, \vv_k = 0\}. \notag
\end{align}
Denoting the sizes of these sets by $a_1, a_2, b_1, b_2, c_1, c_2, d_1, d_2, d_3$, and $e$, respectively, the proposition follows by summing over all possible sets $A_1$, $A_2$, $B_1$, $B_2$, $C_1$, $C_2$, $D_1$, $D_2$, $D_3$, and $E$.
\end{proof}

\begin{proposition}[Lasserre constraint]\label{prop:thirdlasserreqary}
If~$C$ satisfies the inequality $(\lambda_0,\ldots,\lambda_n)\beta$, then for all $a,k$ with~$0 \leq a \leq k \leq n+a-k$ and $k \neq 0$ we have 
\begin{equation}\label{constraints:lasserreqary}
\left(\sum_{t,p=0}^n \alpha(i,j,t,p,a,k) \left(\left(\sum_{d=0}^n\sum_{i',j',t',p'}\lambda_d \eta_{(i',j',t',p'),d}^{(i,j,t,p)} x_{i',j' }^{t',p'}\right) - \beta x^{0,0}_{i+j-t-p,0} \right)\right)_{i,j=k}^{n+a-k} \succeq 0.
\end{equation}
Moreover, 
\begin{equation}\label{constraints:lasserreqarycont}
\begin{pmatrix} \sum_{i=0}^n \binom{n}{i}(q-1)^i \lambda_i x_{0,0}^{0,0} -\beta & y\T \\ y &  T   \end{pmatrix} \succeq 0,
\end{equation}
where 
\begin{align*}
    T=\left(\sum_{t,p=0}^n \alpha(i,j,t,p,0,0) \left(\left(\sum_{d=0}^n\sum_{i',j',t',p'}\lambda_d \eta_{(i',j',t',p'),d}^{(i,j,t,p)} x_{i',j' }^{t',p'}\right) - \beta x^{0,0}_{i+j-t-p,0} \right)\right)_{i,j=0}^{n}, \\
    y_i:=\binom{n}{i}(q-1)^{\frac{1}{2}i}\left(\left(\sum_{d=0}^n\sum_{i',j',t',p'}\lambda_d \eta_{(i',j',t',p'),d}^{(i,i,i,i)} x_{i',j' }^{t',p'}\right) - \beta x^{0,0}_{0,0} \right), \quad \text{for $i=0,\ldots,n$.}
\end{align*} 
\end{proposition}
\begin{proof} 
By Proposition~\ref{prop:third} (i) the matrix $R(c,N)$ is positive semidefinite, where $c=\sum_{\ell=0}^n \lambda_{\ell}\cdot|S_{\ell}(\zero)|\cdot M'_{\zero,\zero} -\beta = \sum_{i=0}^n \binom{n}{i}(q-1)^i \lambda_i x_{0,0}^{0,0} -\beta$. We next compute the entries of~$N$ using Proposition~\ref{prop:third} (iii). We have, for~$\uu,\vv \in \E$, that
\begin{align*}
N_{\uu,\vv}&= -\beta M_{\uu,\vv} + \sum_{d =0}^n \lambda_{d} \sum_{\ww \in S_d(\zero)} M'_{\uu-\ww, \vv-\ww}
\\&= -\beta x_{i+j-t-p,0}^{0,0} +  \sum_{d=0}^n \lambda_d\sum_{\ww \in S_d(\bm 0)} \left(\sum_{i',j',t',p'} x_{i',j'}^{t',p'} M_{i',j'}^{t',p'}  \right)_{\uu-\ww,\vv-\ww} 
 \\&= -\beta x_{i+j-t-p,0}^{0,0}+ \sum_{d=0}^n \lambda_{d} \sum_{\ww \in S_d(\bm 0)} \sum_{i',j',t',p'} x_{i',j'}^{t',p'}\one_{\overline{d}(\uu-\ww,\vv-\ww)=(i',j',t',p')} 
 \\&= -\beta x_{i+j-t-p,0}^{0,0}+ \sum_{d=0}^n \lambda_{d} \sum_{i',j',t',p'} \left( \sum_{\ww \in S_d(\bm 0 )}  \one_{\overline{d}(\uu-\ww,\vv-\ww)=(i',j',t',p')} \right) x_{i',j'}^{t',p'}
 \\&= -\beta x_{i+j-t-p,0}^{0,0} + \sum_{d=0}^n\lambda_d \sum_{i',j',t',p'}\eta_{(i',j',t',p'),d}^{(i,j,t,p)} x_{i',j' }^{t',p'}. 
\end{align*}
Then, Proposition~\ref{prop:RMpsdnonbin} combined with Theorem~\ref{thm:blocknonbin} yields the block-diagonalization of~$R(c,N)$ as given in~\eqref{constraints:lasserreqary},~\eqref{constraints:lasserreqarycont}.
\end{proof} 
It remains to translate the linear constraints from Proposition~\ref{prop:third}~(iv) in terms of the~$x_{i,j}^{t,p}$. To this end, we define the following numbers. Given two words~$\uu$, $\vv \in \E$ with~$\overline{d}(\uu,\vv) = (i,j,t,p)$, let~$\alpha_{(i,j',t',p'),d}^{(i,j,t,p)}$ denote the number of words~$\ww \in \E$ with $\overline{d}(\uu,\ww)=(i,j',t',p')$ and~$d(\vv,\ww) =d$. As in the binary case, the numbers~$\alpha_{(i,j',t',p'),d}^{(i,j,t,p)}$ can be explicitly calculated as sums of expressions including binomial coefficients.
\begin{lemma}\label{lem:alphaqary}
 The numbers~$\alpha_{(i,j',t',p'),d}^{(i,j,t,p)}$ satisfy
\begin{align}
\alpha_{(i,j',t',p'),d}^{(i,j,t,p)} =  \sum_{\substack{a_1,a_2,\\b_1,b_2,\\c_1,c_2,\\d_1,d_2,d_3,\\e}} \binom{i-t}{a_1,a_2}&\binom{j-t}{b_1,b_2}\binom{p}{c_1,c_2}\binom{t-p}{d_1,d_2,d_3}\binom{n+t-i-j}{e}\notag\\&\cdot(q-1)^e (q-2)^{a_2+b_2+c_2}(q-3)^{d_3},
\end{align}
where~$a_1$, $a_2$, $b_1$, $b_2$, $c_1$, $c_2$, $d_1$, $d_2$, $d_3$, and $e$ range over all nonnegative integers for which
\begin{align}
e &\leq n+t-i-j, \,\, b_1+b_2 \leq j-t, \,\, a_1+a_2 \leq i-t, \,\, d_1+d_2+d_3 \leq t-p, \,\, c_1+c_2 \leq p\\ 
j' &= a_1+a_2+b_1+b_2+c_1+c_2+d_1+d_2+d_3+e \notag\\
t'&= a_1+a_2+c_1+c_2+d_1+d_2+d_3\notag\\
p'&= a_1+c_1+d_1\notag\\
d&= a_1+a_2+e+j-b_1-c_1-d_2.\notag
\end{align}
\end{lemma}
\begin{proof} 
Analogous to Lemma~\ref{lem:etaqary}.
\end{proof} 

\begin{lemma}\label{lem:lambdasqary}
Let $\uu \in \E$ with $d(\uu,\zero)=i$ and let $x \in \mathbb{R}^{\E}$ be such that $x_{\vv}$ is determined by $\overline{d}(\uu ,\vv)$, say $x_{\vv}=x_{i,j}^{t,p}$, in case $\overline{d}(\uu ,\vv)=(i,j,t,p)$. Then 
\[
\sum_{d=0}^n \lambda_d x(S_d( \vv ))\geq \beta \quad \text{for all $\vv \in\E$} \quad \Longleftrightarrow \quad
\sum_{j',t',p'} x_{i,j'}^{t',p'}\cdot \sum_{d=0}^n \lambda_d \alpha_{(i,j',t',p'),d}^{(i,j,t,p)} \geq \beta \quad \text{for all $j,t,p$},
\]
where $x(S_d( \vv )) := \sum_{\ww \in S_d(\vv)} x_{\ww}$.
\end{lemma}
\begin{proof}
Let $\vv \in \E$ and $\overline{d}(\uu,\vv) = (i,j,t,p)$. Then we have 
\begin{align*}
\sum_{d=0}^n \lambda_d x(S_d(\vv)) & = \sum_{d=0}^n \lambda_d \sum_{\substack{\ww \in\E \\ d(\vv,\ww) = d}} x_{\ww} = \sum_{d=0}^n \lambda_d \sum_{j',t',p'} \sum_{\substack{\ww \in \E\\d(\vv, \ww)=d \\ \overline{d}(\uu, \ww)=(i,j',t',p')}}x_{\ww} \\
&= \sum_{d=0}^n \lambda_d\sum_{j',t',p'} \alpha_{(i,j',t',p'),d}^{(i,j,t,p)} x_{i,j'}^{t',p'}= \sum_{j',t'} x_{i,j'}^{t',p'} \cdot \sum_{d=0}^n \lambda_d \alpha_{(i,j',t',p'),d}^{(i,j,t,p)}. \qedhere
\end{align*}
\end{proof}

Now we are ready to rewrite the linear constraints from Proposition~\ref{prop:third} (iv) in terms of the~$x_{i,j}^{t,p}$. 

\begin{proposition}[Matrix cut inequalities]\label{prop:thirdmatrixcutqary} If~$C$ satisfies the inequality $(\lambda_0,\ldots,\lambda_n)\beta$, then for every tuple~$(i,j,t,p) \in I(q,n)$ the following inequalities hold:
\begin{align} \label{constraints:matrixcutqary}
    \sum_{j',t',p'} x_{i,j'}^{t',p'} \cdot \lambda_{(i,j',t',p')}^{(i,j,t,p)} & \geq x_{i,0}^{0,0} \beta, \\
    \sum_{j',t',p'} (x_{j',0}^{0,0} - x_{i,j'}^{t',p'}) \cdot \lambda_{(i,j',t',p')}^{(i,j,t,p)}& \geq (x_{0,0}^{0,0} -  x_{i,0}^{0,0}) \beta, \notag\\
    \sum_{j',t',p'} (x_{i+j'-t'-p',0}^{0,0} - x_{i,j'}^{t',p'} )\cdot \lambda_{(i,j',t',p')}^{(i,j,t,p)} & \geq (x_{0,0}^{0,0} -  x_{i,0}^{0,0}) \beta,\notag\\
    \sum_{j',t',p'} (x_{0,0}^{0,0} - x_{j',0}^{0,0} -x_{i+j'-t'-p',0}^{0,0} + x_{i,j'}^{t',p'} )\cdot \lambda_{(i,j',t',p')}^{(i,j,t,p)} & \geq (1-2x_{0,0}^{0,0} +  x_{i,0}^{0,0}) \beta, \notag
\end{align}
where the numbers~$\lambda_{(i,j',t',p')}^{(i,j,t,p)}$ are given by
\[
\lambda_{(i,j',t',p')}^{(i,j,t,p)} := \sum_{d=0}^n \lambda_d \alpha_{(i,j',t',p'),d}^{(i,j,t,p)}.
\]
\end{proposition}
\begin{proof} 
These are the constraints from Proposition~\ref{prop:third} (iv) written in terms of the~$x_{i,j}^{t,p}$, making use of Lemma~\ref{lem:lambdasqary}. 
\end{proof}

The following theorem makes the semidefinite programming lower bound from Theorem~\ref{thm:originalbound} effective using symmetry reduction for general~$q$-ary codes, that is, it gives an efficient semidefinite programming lower bound on~$K_q(n,r)$ for arbitrary~$q$.
\begin{theorem} \label{thm:sdplowerboundqary}
If every code~$C \subseteq \mathbb{E}$ with covering radius~$r$ satisfies $(\lambda_0,\lambda_1,\ldots,\lambda_n)\beta$, then we have
\begin{align}\label{eq:objectiveqary}
    K_q(n,r)^3 \geq \min_x q^n\sum_{(i,j,t,p)\in I(q,n)} (q-1)^{i+j-t}(q-2)^{t-p}\binom{n}{p,t-p,i-t,j-t}  x_{i,j}^{t,p} , 
\end{align}
where the minimum ranges over all real~$x=(x_{i,j}^{t,p})$ satisfying the constraints from Propositions~\ref{prop:firstqary}, \ref{prop:secondqary}, \ref{prop:thirdlasserreqary}, and  \ref{prop:thirdmatrixcutqary}.
\end{theorem}
\begin{proof}
 Equation~$\eqref{eq:objectiveqary}$ is the objective function from Theorem~\ref{thm:originalbound} written in terms of the~$x_{i,j}^{t,p}$.
 The validity of the constraints from Propositions~\ref{prop:firstqary},~\ref{prop:secondqary},~\ref{prop:thirdlasserreqary}, and~\ref{prop:thirdmatrixcutqary} was already proven.   
 \end{proof} 

By~\eqref{dimAqn}, both the number of variables of the reduced SDP and the sum of squares of its block sizes are $O(n^4)$.

\section{Tables with the new lower bounds\label{sec:tables}}

We used \texttt{Julia} to generate the semidefinite programs. We also developed an independent \texttt{Perl} code to generate the SDPs, which we used to validate our Julia code. Most SDPs were solved directly in \texttt{Julia} via \texttt{JuMP} with the high-precision solver \texttt{SDPA-GMP}~\cite{SDPAGMP}. The solving time varied from a few seconds to several days for the larger instances in the tables below. The precision in \texttt{SDPA-GMP} was set to 512 bits, and custom parameters (mainly a larger lambdaStar, the parameter which determines an initial point for the interior point method) were used as opposed to the standard parameter settings provided with \texttt{SDPA-GMP} to ensure convergence for the larger instances.

Generating the SDP for larger instances posed additional challenges due to large coefficients, necessitating the use of high-precision numbers (we used 512 bits BigFloats). Our code for generating the SDP including the call to the solver \texttt{SDPA-GMP} and the parameters used for most of the instances we solved, is available here:  
\begin{center}
\url{https://github.com/CoveringCodes}
\end{center}

In the following tables, we present our new bounds obtained from the semidefinite programs in Theorem~\ref{thm:sdplowerbound} and \ref{thm:sdplowerboundqary}, and compare them to the best known bounds from Keri's tables~\cite{KeriTables}. While we also find improved bounds of~$K_2(18,1) \geq 14665$ and~$K_2(30,1) \geq 35874398$, these have recently been surpassed by Wu and Chen~\cite{WuChen24}, who established~$K_2(18,1) \geq 14666$ and~$K_2(30,1) \geq 35876816$. Therefore we exclude these bounds from Table~\ref{binarytable}.

\begin{table}[H] \small
\begin{center}
   \begin{minipage}[t]{.47\linewidth}\centering
    \begin{tabular}{| r | r|| r |>{\bfseries}r | r|}
    \hline
   $n$ & $R$ & \multicolumn{1}{b{14mm}|}{best lower bound previously known}  & \multicolumn{1}{b{11mm}|}{\textbf{new lower bound}} &  \multicolumn{1}{b{10mm}|}{best upper bound known}\\\hline 
      13 & 1  & 598 & 607  & 704  \\       
      14 & 1  & 1172 & 1185  & 1408   \\         
      17 & 1  & 7419 & 7426  & 8192  \\       
      21 & 1  & 96125 & 96477  &  122880 \\       
      22 & 1  & 190651 & 191501  & 245760  \\    
      25 & 1  & 1298238 & 1301089  & 1556480  \\       
      26 & 1  & 2581111 & 2589179  &  3112960 \\         
      29 & 1  & 17997161 & 18000844  & 23068672  \\       
      33 & 1  & 253523901 & 253764801  & 268435456  \\         \hline 
      13 & 2  & 97 & 101   & 128  \\       
      14 & 2  & 159 & 170  & 248   \\      
      17 & 2  & 859 & 889  & 1536  \\       
      23 & 2  & 30686 & 30828  & 32768   \\  
      26 & 2  &  191229 & 192747  & 262144  \\       
      29 & 2  &   1231554 &  1239885 & 2097152  \\  
      32 & 2  & 8170308 &  8173960 & 16776960  \\\hline
      12 & 3  & 18 & 19 & 28  \\       
    \hline 
    \end{tabular}\end{minipage}    
   \begin{minipage}[t]{.47\linewidth}\centering
    \begin{tabular}{| r | r|| r |>{\bfseries}r | r|}
    \hline
   $n$ & $R$ & \multicolumn{1}{b{14mm}|}{best lower bound previously known}  & \multicolumn{1}{b{11mm}|}{\textbf{new lower bound}} &  \multicolumn{1}{b{10mm}|}{best upper bound known}\\\hline  
      27 & 3  & 40683 & 41012  & 65536   \\\hline          
      15 & 4  & 22 & 23  & 32 \\       
      16 & 4  & 33 & 34  & 64 \\   \hline  
      16 & 5  & 13 & 14  & 28 \\       
      17 & 5  & 19 & 20  & 32  \\    
      18 & 5  & 27 & 28  & 64   \\    \hline    
      18 & 6  & 12 & 13  & 28  \\         
      19 & 6  & 16 & 17  & 32  \\       
      21 & 6  & 33 & 34  & 64  \\    \hline 
      21 & 7  & 14 & 15  & 32  \\         
      22 & 7  & 20 & 21  & 64  \\  \hline      
      23 & 8  & 13 & 14  & 32   \\      
      24 & 8  & 18 & 19  & 64  \\   \hline     
      25 & 9  & 12 & 13  & 32  \\  
      26 & 9  & 16 & 17  & 56  \\    \hline    
      27 & 10  & 11 & 12 & 32   \\  
      29 & 10  & 19 & 20 & 64  \\             
    \hline 
    \end{tabular}\end{minipage}    
\end{center}
  \caption{New lower bounds on $K_2(n,R)$}\label{binarytable}
\end{table}

\begin{table}[H] \small
\begin{center}
   \begin{minipage}[t]{.47\linewidth}\centering
    \begin{tabular}{| r | r|| r |>{\bfseries}r | r|}
    \hline
   $n$ & $R$ & \multicolumn{1}{b{14mm}|}{best lower bound previously known}  & \multicolumn{1}{b{11mm}|}{\textbf{new lower bound}} &  \multicolumn{1}{b{10mm}|}{best upper bound known}\\\hline 
       8 & 1   & 402 & 403 & 486  \\           
       9 & 1  & 1060 & 1064  & 1269   \\  \hline      
       7 & 2  & 26 & 27 & 34  \\           
       8 & 2  & 54 & 58 & 81  \\       
       9 & 2  & 130 & 132 & 219  \\           
       14& 2   & 12204 &12323  & 19683  \\  \hline     
       8& 3  & 14 & 16  & 27  \\           
       9 &3   &27  & 31 & 54  \\      
       10& 3  &57  & 61 & 105  \\       
       11& 3  & 117 & 129 & 243  \\           
       13& 3  & 612 & 640 & 1215   \\           
    \hline 
    \end{tabular}\end{minipage}    
   \begin{minipage}[t]{.47\linewidth}\centering
    \begin{tabular}{| r | r|| r |>{\bfseries}r | r|}
    \hline
   $n$ & $R$ & \multicolumn{1}{b{14mm}|}{best lower bound previously known}  & \multicolumn{1}{b{11mm}|}{\textbf{new lower bound}} &  \multicolumn{1}{b{10mm}|}{best upper bound known}\\\hline 
       10& 4  & 17  & 18 & 36  \\           
       11& 4  & 30 & 34 &  81 \\       
       12& 4  & 62 & 65 & 175  \\           
       13& 4  & 123 & 130 & 335  \\       
       14&4   & 255 & 273 & 729  \\\hline            
       11&5   & 11 &  12&  27 \\       
       12&5   & 18 &  21&  54 \\           
       13&5   &  33& 37 & 108  \\      
       14&5   & 59 & 69 & 243  \\ \hline      
       13&6   & 13 &  14 & 36  \\           
       14&6   & 21 & 24 & 81  \\         
    \hline 
    \end{tabular}\end{minipage}    
\end{center}
  \caption{New lower bounds on $K_3(n,R)$}
\end{table}

\begin{table}[H] \small
\begin{center}
   \begin{minipage}[t]{.47\linewidth}\centering
    \begin{tabular}{| r |r | r|| r |>{\bfseries}r | r|}
    \hline
   $q$ & $n$ & $R$ & \multicolumn{1}{b{14mm}|}{best lower bound previously known}  & \multicolumn{1}{b{11mm}|}{\textbf{new lower bound}} &  \multicolumn{1}{b{10mm}|}{best upper bound known}\\\hline 
    4 &  7 & 1  & 762 & 776 & 992  \\    
    4 &  11 & 1  &123846 & 124941  & 131072  \\\hline          
    4 & 6 & 2  & 32& 33  & 52  \\    
    4 &  7 & 2  &84 &88  & 128  \\     
    4 & 8 & 2  & 240 & 251 & 352   \\    
    4 &  9 & 2  & 751 & 775 & 1024  \\   
    4 & 10 & 2  &2412  &2460  & 4096   \\    
    4 &  11 & 2  & 7974 & 8072 &  15872  \\\hline        
    4 &  8 & 3  & 44 & 46 &  96 \\   
    4 & 9 & 3  & 110 & 116  & 256  \\    
    4 &  11 & 3  & 849 & 885 & 2048  \\ \hline 
    4 &  9  & 4  & 26& 27 & 64  \\  
    4 &  10  & 4  & 59 & 62 & 208  \\  
    4 &  11  & 4  & 148 & 150 & 512  \\\hline   
    4 & 11  & 5  & 36 & 37 & 128  \\      
    \hline 
    \end{tabular}\end{minipage}    
   \begin{minipage}[t]{.47\linewidth}\centering
    \begin{tabular}{| r |r | r|| r |>{\bfseries}r | r|}
    \hline
   $q$ & $n$ & $R$ & \multicolumn{1}{b{14mm}|}{best lower bound previously known}  & \multicolumn{1}{b{11mm}|}{\textbf{new lower bound}} &  \multicolumn{1}{b{10mm}|}{best upper bound known}\\\hline 
    5 &  5 & 1  & 160 & 162& 184  \\    
    5 &  7 & 1  & 2722 & 2765 & 3125  \\          
    5 &  8 & 1  & 11945 &12134 & 15625  \\    
    5 &  9 & 1  & 53138 & 53896&  78125 \\    
    5 &  10 & 1  & 238993 &241122 & 390625  \\  \hline  
    5 &  7 & 2  & 225 &236 & 525  \\          
    5 &  8 & 2  & 821 &861 & 1625  \\    
    5 &  11 & 2   & 52842 & 53309& 78125  \\  \hline        
    5 &  8  & 3  & 109 & 111& 325  \\    
    5 & 9  & 3  & 330 & 354&  1275  \\          
    5 & 10  & 3  & 1163 & 1215& 3125  \\    
    5 &  11 & 3  & 4255 & 4366& 15625  \\ \hline 
    5 &  10 &  4 &  162& 177&  875 \\    
    5 &  11 &  4 &  535& 546& 3125  \\ \hline    
    \end{tabular}\end{minipage}     
\end{center}
  \caption{New lower bounds on $K_4(n,R)$ and $K_5(n,R)$}
\end{table}

To illustrate the size of the computations, we report the number of variables in the final symmetry-reduced SDP, together with the sum of the block sizes in the block-diagonalization of $M'$ and the sum of the squares of these block sizes, see Table~\ref{table:sdpsizes}. We note that the dimensions of the Terwilliger algebras are $\tbinom{n+3}{3}$ in the binary case, and $\tbinom{n+4}{4}$ in the nonbinary case, but the actual number of variables in our SDP is smaller, since the additional symmetry relations from Propositions~\ref{prop:firstbinary} (iii) and \ref{prop:firstqary} (iii) are imposed.

By Theorems~\ref{thm:blockhammingcube} and~\ref{thm:blocknonbin}, the block sizes in the block diagonalization of $M'$ are given by $n-2k+1$ for $k=0,\ldots,\lfloor n/2 \rfloor$ in the binary case, and $n+a-2k+1$ for all pairs $(a,k)$ with $0 \le a \le k \le  n+a-k $ in the nonbinary case. 
The matrices $M''$ and $N$ have the same block structure after block-diagonalization, except that the block corresponding to $k=0$ or $(a,k)=(0,0)$ is bordered, increasing its size by one. The table reports the block sizes after symmetry reduction for the matrix $M'$. The full SDP contains several semidefinite constraints (for $M'$, $M''$, and one or more matrices $N$), each with essentially the same block sizes. Thus, the total size of the SDP is a small constant multiple of the values reported in the table.

The computational effort also depends on the choice of inequalities $(\lambda_0,\ldots,\lambda_n)\beta$, since the construction of the entries of the matrices $N$ involves combinatorial sums whose number of terms increases with the number of nonzero coefficients $\lambda_i$. Consequently, for larger $R$ the SDP takes much longer to generate, and typically also to solve, even though the block sizes and the number of variables remain unchanged for fixed $n$.

\begin{table}[H]\small
\centering
\begin{tabular}{c c c c c}
\hline
$q$ & $n$ & \#variables & $\sum \text{block sizes}$ & $\sum (\text{block size})^2$ \\
\hline
2 & 12 & 102 & 49 & 455 \\
2 & 22 & 458 & 144 & 2300 \\
2 & 32 & 1239 & 289 & 6545 \\
\hline
3 & 8 & 136 & 95 & 495 \\
3 & 11 & 339 & 203 & 1365 \\
3 & 14 & 711 & 372 & 3060 \\
\hline 
4, 5 & 6 & 64 & 50 & 210 \\
4, 5 & 11 & 339 & 203 & 1365 \\
\hline
\end{tabular}
\caption{\small Size parameters of the symmetry-reduced SDPs. We report the number of variables, the sum of the block sizes, and the sum of the squares of the block sizes for the block-diagonalization of $M'$.}
\label{table:sdpsizes}
\end{table}

The total running time (including both SDP generation and solving) ranged from a few minutes for smaller instances to several days on standard desktop hardware, and in some cases up to two weeks for the largest instances. These runtimes depend strongly on the parameters $(q,n,R)$, the numerical precision, and solver settings. In particular, we used \texttt{SDPA-GMP}, as the SDPs can have very large coefficients and appear to require high-precision arithmetic for reliable solution. The reported runtimes should therefore be interpreted as indicative. All SDPs could very comfortably be generated and solved on a machine with 32\,GB of RAM, although the generation and solving of larger instances required several gigabytes due to the use of high-precision arithmetic.

\appendix 
\section{All computational results for~\texorpdfstring{$q=2$}{q=2}\label{appendix:tablesbinary}}
We computed the lower bounds following from Theorem~\ref{thm:sdplowerbound} using the two $(\lambda_0,\ldots,\lambda_n)\beta$ coming from the sphere covering inequalities and the Van Wee inequalities. So positive semidefiniteness of the Lasserre matrices and the matrix cut inequalities for both $(\lambda_0,\ldots,\lambda_n)\beta$ are included as constraints. In Tables~\ref{table:binarysmallr} and~\ref{table:binarylarger} the resulting lower bounds are shown. Improvements over the values in K\'eri's tables with bounds for covering codes (cf.~\cite{KeriTables}) are in boldface and marked with an asterisk~$(*)$.\footnote{The improved bounds of~$K_2(18,1) \geq 14665$ and~$K_2(30,1) \geq 35874398$ have recently been surpassed by Wu and Chen~\cite{WuChen24}, who established~$K_2(18,1) \geq 14666$ and~$K_2(30,1) \geq 35876816$. They are not marked in boldface here.}

\begin{table}[H]\centering {\small\begin{tabular}{r| r*{6}{r}}\hline
    $n$    &   $R=1$ & $R=2$ & $R=3$ & $R=4$ & $R=5$ &$R=6$\\\hline
4 & 3.9999 & X& X&X & X&X \\
5 & 6.6721 & X& X &X & X&X \\
6 & 11.5980 &X &X & X& X& X\\
7 & 15.9999 &X &X & X&X & X\\
8 & 31.9999 &  9.5889 & X&X &X &X \\
9 & 55.3464 & 14.7583 &X &X &X &X \\
10 & 105.2223 & 22.4103 & X&X &X &X \\
11 & 170.6666 & 35.5187&12.4700 & X&X & X\\
12 & 341.3333 & 61.2153 & \textbf{18.6887*} & 7.9873& X&X \\
13 & \textbf{606.7119*} & \textbf{100.2419*}& 27.6830 & 11.1057 &X &X \\
14 & \textbf{1184.7592*}& \textbf{169.0859*} &42.6133 &  15.7629 & 7.3889 &X \\
15 & 2047.9999 & 304.7919& 67.0893 & \textbf{22.6403*}& 10.0206 & X\\
16 & 4095.9999 & 511.8192& 112.5592 & \textbf{33.2584*} & \textbf{13.7867*} & 6.9175  \\
17 & \textbf{7425.1563*} & \textbf{888.3163*}& 179.1108 &49.9774 & \textbf{19.2500*} &9.1701\\
18 & 14664.0012 & 1691.3396& 287.5958 & 76.5456& \textbf{27.4903* } & \textbf{12.3285*}\\
19 & 26214.3999 & 2887.1523& 479.1587  & 119.7022&  39.8079 & \textbf{16.8119*}\\
20 & 52428.7999 & 5106.9177 &  881.5640 &  201.5243& 58.7224 & 23.2979\\
21 & \textbf{96476.0879*} & 9866.5048 &  1465.5810 & 316.8379 &88.1658 & \textbf{33.2651*}\\
22 & \textbf{191500.9383*} & 17285.1830& 2453.8583  & 504.9234 &134.8236 & 47.7532 \\
23 & 349525.3333 & \textbf{30827.8234*}& 4095.9999  & 815.4795 &209.2984 & 69.5451\\
24 & 699050.6666 & 60159.4032 & 8006.4570 &1352.0812 &356.0391 & 102.8095 \\ 
25 & \textbf{1301088.8409*} & 107046.2406 & 13697.8925& 2497.6197&559.2092  & 154.7948\\
26 & \textbf{2589178.6211*}&\textbf{192746.8483*} &  23573.8897& 4120.3497&886.0069 & 236.5925\\
27 & 4793490.2857  & 379479.9404& \textbf{41011.0436*}& 6853.9216& 1415.7917 & 366.6334 \\
28 & 9586980.5714 & 683131.0859 & 79733.418&11484.4883 & 2283.7356& 629.5189\\
29 & \textbf{18000843.7204*} & \textbf{1239884.9874*}& 139378.2397 &19483.6215 & 3768.8189 & 983.7626\\
30 & {35874397.2504} & 2450496.1520& 244591.6475 & 37401.0004 & 6982.5328 &  1550.4459\\
31 & 67108863.9999 & 4459646.6807& 432709.8573  & 63716.8859 & 11505.1662 & 2463.3423\\
32 &  134217727.9999 & \textbf{8173959.5702*}& 847720.8613& 109182.7236& 19067.0734 & 3942.9567\\
33 &  \textbf{253764800.5874*}& 16151563.3396&  1504923.6064& 187950.9643 & 31776.3316 & 6358.3611\\
\hline
\end{tabular}  }
\caption{\label{table:binarysmallr}Bounds for~$q=2$ and~$R \leq 6$ }
\end{table}

\begin{table}[H]\centering {\small\begin{tabular}{r| r*{4}{r}}\hline
    $n$    &   $R=7$ & $R=8$ & $R=9$ & $R=10$\\\hline
18 & 6.5517& X& X& X \\
19 & 8.4991 & X & X& X \\
20 & 11.2018& 6.2720&  X & X \\
21 & \textbf{14.9725*}& 7.9896 & X &X  \\
22 & \textbf{20.3253*}& 10.3360& 6.0509& X \\
23 & 28.0646& \textbf{13.5670*} & 7.5738 &X  \\
24 & 40.1306 & \textbf{18.0802*} & 9.6469 &5.8563  \\
25 & 57.2070& 24.4924& \textbf{12.4581*} & 7.2502 \\
26 & 82.6535&33.6771 & \textbf{16.3383*}& 9.1059 \\
27 &121.0256 &48.3474 &21.7709 & \textbf{11.5788*} \\
28 & 179.6661& 68.6061&29.4411 & 14.9654 \\
29 & 270.8681& 98.5932 &40.3898 & \textbf{19.6353*}  \\
30 &413.8114 & 143.3834& 58.2527&  26.1561\\
31 & 640.0524& 210.9974& 82.3517& 35.3432  \\
32 &1108.7839 & 314.1767 & 117.8229& 48.4397 \\
33 & 1728.0332&473.3013 &  170.4590& 70.2228 \\
\hline
\end{tabular}}
\caption{\label{table:binarylarger}Bounds for~$q=2$ and~$R\geq 7$ }
\end{table}

\section{Computational results for nonbinary covering codes \label{appendix:tablesqary}} 
In this section we present bounds on~$K_q(n,r)$ for~$q =3,4,5$. Tables~\ref{table:3ary},~\ref{table:4ary} and~$\ref{table:5ary}$ contain the lower bounds obtained with the SDP from Theorem~\ref{thm:sdplowerboundqary}. We again used the $(\lambda_0,\ldots,\lambda_n)\beta$ coming from the sphere covering inequalities \textbf{but not the Van Wee inequalities}. Improvements over the values in K\'eri's tables with bounds for covering codes (cf.~\cite{KeriTables}) are in boldface and marked with an asterisk~$(*)$. 
\begin{table}[H]\centering {\small\begin{tabular}{r| r*{8}{r}}\hline
    $n$    &   $R=1$ & $R=2$ & $R=3$& $R=4$  & $R=5$ & $R=6$ & $R=7$ \\\hline
6 & 60.8568   &13.1228   & X& X & X& X& X\\
7 & 150.9556   & \textbf{26.3830*}& 8.5250 & X & X& X& X \\
8 & \textbf{402.9463*} & \textbf{57.4972*}& \textbf{15.5959*} &  X &X &X &X\\
9 & \textbf{1063.9751*} & \textbf{131.8916*}&  \textbf{30.1035*}& 10.2323  & X& X&X\\
10 & 2811.8571 & 313.0763 & \textbf{60.4226*} & \textbf{17.9976*}  & 7.3464& X&X\\
11 & 7823.1641 & 728.9999& \textbf{128.2649*} & \textbf{33.2215*}  & \textbf{11.9665*}&X &X\\
12 & 21527.7873 & 1897.3134 & 281.6598 & \textbf{64.0810*}   &\textbf{20.5195*} & 8.5946&X\\
13 & 59048.9999& 4802.7269& \textbf{639.2731*} & \textbf{129.4563*}  & \textbf{36.7143*}&\textbf{13.8047*} & 6.5667\\
14 & 166229.5226 & \textbf{12322.3736*} & 1500.5419& \textbf{272.0850*} & \textbf{68.5803*} & \textbf{23.2533*} & 9.9504\\
\hline \rowcolor{gray!30}  
15 & 466049.0035& 32148.5550& 3566.1881& 591.2872 & 134.0645 & 40.7700& 15.8151 \\\rowcolor{gray!30}  
16 & 1304446.0909 & 84390.4101 & 8765.7218 &1319.1922 &272.2170 & 74.4488 &26.2497
\\\hline
\end{tabular}  }
\caption{\label{table:3ary}Bounds for $q=3$ }
\end{table}

\begin{table}[H]\centering {\small\begin{tabular}{r| r*{6}{r}}\hline
    $n$    &   $R=1$ & $R=2$ & $R=3$ & $R=4$& $R=5$& $R=6$  & $R=7$\\\hline
6 &226.59 &\textbf{32.91*} &  8.76 &  3.35 &X &X & X\\
7 &\textbf{775.07*}&\textbf{87.63*} &  18.78 & 6.40& X&X  &X \\
8 &2694.38 & \textbf{250.87*}& \textbf{45.02*} &12.56 &X &X  &X\\
9 &9362.28 &\textbf{774.46*} & \textbf{115.28*} & \textbf{26.66*} & 9.03& X &X\\
10 &34254.96 & \textbf{2459.70*}& 310.69 & \textbf{61.09*} & 17.54  &  6.87 &X\\
11 & \textbf{124940.31*}& \textbf{8071.19*}&\textbf{884.35*} & \textbf{149.89*}& \textbf{36.85*} & 12.48 & 5.52\\\hline \rowcolor{gray!30}
12 & {457543.61} & {27020.63}& {2621.72} & {390.39}& {82.75} & {24.20}& 9.36\\\rowcolor{gray!30}
13 & {1677721.59}  &91311.99 & 8046.26 &  1065.18& 198.06& 50.50& 17.05\\
\hline
\end{tabular} } 
\caption{\label{table:4ary}Bounds for $q=4$ }
\end{table}

\begin{table}[H]\centering{\small \begin{tabular}{r| r*{7}{r}}\hline
    $n$    &   $R=1$ & $R=2$ & $R=3$ & $R=4$ & $R=5$ & $R=6$ & $R=7$ \\\hline
5 & \textbf{161.03*} &  21.66 & X& X&X &X & X\\   
6 & 624.99 &  68.86 & 13.81&  4.37&X &X & X\\
7 & \textbf{2764.89*} &\textbf{235.35*} & 37.40 & 9.70& X&X & X\\
8 &\textbf{12133.70*} &\textbf{860.13*} & \textbf{110.24*} &  23.04& 7.27&X &X \\
9 &\textbf{53895.34*} & 3279.51& \textbf{353.32*}&61.18 & 15.67 & 5.79 & X\\
10 &\textbf{241121.95*} & 13060.70 & \textbf{1214.11*} & \textbf{176.26*} & 37.81& 11.40& X\\
11 & 1085069.44& \textbf{53308.26*} &  \textbf{4365.35*}& \textbf{545.18*}& 99.28& 25.20 & 8.81 \\
\hline \rowcolor{gray!30}
12 & 5013157.53 & 222134.17& {16309.71}& 1789.77& 280.94& 61.34& 17.89\\ \rowcolor{gray!30}
13 &  23209474.92& 943510.75& 62819.97& 6165.15& 849.14&160.09 &40.59 \\\hline
\end{tabular}  }
\caption{\label{table:5ary}Bounds for $q=5$ }
\end{table}

\end{document}